\documentclass[10pt]{article}

\usepackage[utf8x]{inputenc}
\usepackage{amsfonts}
\usepackage{amsmath}\usepackage{amssymb}
\usepackage{amsthm}
\usepackage{mathrsfs}
\usepackage{latexsym}
\usepackage{graphicx}
\usepackage{bm}
\usepackage{inputenc}
\usepackage[shortlabels]{enumitem}

\usepackage{hyperref}
\usepackage{graphicx}
\usepackage[a4paper,bindingoffset=0.2in,%
left=1.0in,right=1.0in,top=1.2in,bottom=1.2in,%
footskip=.35in]{geometry}

\allowdisplaybreaks

\swapnumbers

\newtheorem{Lemma}{Lemma}[section]
\newtheorem{Corollary}[Lemma]{Corollary}
\newtheorem{Theorem}[Lemma]{Theorem}
\newtheorem{Conjecture}[Lemma]{Conjecture}

\theoremstyle{definition}

\newtheorem{Definition}[Lemma]{Definition}

\theoremstyle{remark}

\newtheorem{Remark}[Lemma]{Remark}

\newtheoremstyle{citing}
{3pt}
{3pt}
{\itshape}
{}
{\bfseries}
{.}
{.5em}
{\thmnote{#3}}
\theoremstyle{citing}

\newtheoremstyle{proof*}
{3pt}
{3pt}
{\rmfamily}
{}
{ \itshape}
{.}
{.5em}
{\thmnote{#3}}
\theoremstyle{proof*}
\newtheorem*{proof*}{}

\DeclareMathOperator{\restrict}{\llcorner}

\DeclareMathOperator{\Tan}{Tan}     
\DeclareMathOperator{\diam}{diam}

\DeclareMathOperator{\Unp}{Unp}  
\DeclareMathOperator{\dmn}{dmn}     
\DeclareMathOperator{\Nor}{Nor}
\DeclareMathOperator{\Hom}{Hom}     
\DeclareMathOperator{\Der}{D}

\DeclareMathOperator{\Clos}{Clos}
\DeclareMathOperator{\ap}{ap}

\DeclareMathOperator{\Cut}{Cut}

\DeclareRobustCommand{\rchi}{{\mathpalette\irchi\relax}}
\newcommand{\irchi}[2]{\raisebox{\depth}{$#1\chi$}}

\usepackage{color}

\title{Anisotropic curvature measures and uniqueness of convex bodies}
\author{Mario Santilli}
\date{}

\begin{document}
	\maketitle
	
	\begin{abstract}
		We prove that an arbitrary convex body $ C \subseteq \mathbf{R}^{n+1} $, whose $ k $-th anisotropic curvature measure (for $ k =0, \ldots , n-1 $) is a multiple constant of the anisotropic perimeter of $ C $, must be a rescaled and translated Wulff shape. This result provides a generalization of a theorem of Schneider (1979) and resolves a conjecture of Andrews and Wei (2017).
	\end{abstract}
	
	\section{Introduction}
	
	\paragraph{} Jellett (1853) proved that a compact embedded star-shaped hypersurface of the Euclidean space with constant mean curvature must be a round sphere. A century later Hsiung (see \cite{MR68236}), extending a result of Süss (1929), proved that the same conclusion holds for compact embedded star-shaped hypersurfaces with constant higher-order mean curvatures. The fundamental results of Alexandrov (see \cite{MR143162}), Ros (see \cite{MR996826}), Korevaar-Ros (see \cite{MR925120}) and Montiel-Ros (see \cite{MR1173047}) guarantee that the same results hold without assuming that the hypersurfaces are star-shaped. These theorems, as well as their method of proof, lie at the core of modern differential geometry and its applications. 
	
	\paragraph{} It is a natural question to extend these type of results to arbitrary convex bodies. Of course, in this setting one has to carefully choose the notion of curvature, in order to handle the unavoidable singular set. An insightful notion of curvature has been introduced by Federer in \cite{MR0110078} with the concept of \emph{curvature measures}. If $ C \subseteq \mathbf{R}^{n+1} $ is an arbitrary convex body, we denote with $ \bm{\delta}_C$ the distance function from $ C $ and with $ \bm{\xi}_C $ the metric projection onto $ C $. Then there exist uniquely determined Radon measures $ \mathcal{C}_0(K, \cdot), \ldots , \mathcal{C}_n(C, \cdot) $ supported on the boundary $ \partial C $ of $ C $ such that 
	\begin{equation*}
		\mathcal{L}^{n+1}(\{ x \in \mathbf{R}^{n+1}: 0 < \bm{\delta}_C(x)\leq \rho, \; \bm{\xi}_C(x) \in B  \}) = \sum_{m=0}^{n} \rho^{n+1-m}\mathcal{C}_m(C, B)
	\end{equation*}
	for every Borel subset $ B \subseteq \partial C $. For an arbitrary convex body the equality $ \mathcal{C}_n(C, \cdot) = \mathcal{H}^n \restrict \partial C $ always holds.	If $ \partial C $ is a $ \mathcal{C}^2 $-hypersurface, then 
	\begin{equation*}
		\mathcal{C}_m(C, B) = \frac{1}{n+1-m}\int_{B}H_{n-m}\, d\mathcal{H}^n \qquad \textrm{for $ m =0, \ldots , n $ and $B \subseteq \partial C $,}
	\end{equation*}
	where 
	\begin{equation*}
		H_k = \sum_{1 \leq l_1 < \ldots < l_k \leq n}\kappa_{l_1} \cdots \kappa_{l_k}
	\end{equation*}
	denotes the $ k $-th mean curvature of $ \partial C $ (with $ H_0 \equiv 1 $). The aforementioned uniqueness results for hypersurfaces with constant $ k $-th mean curvatures have been extended to arbitrary convex bodies by Schneider in 1979.
	\begin{Theorem}[\protect{cf.\ \cite{MR522031}}]\label{theo: schneider}
		If $ C \subseteq \mathbf{R}^{n+1} $ is an arbitrary convex body, $ m = 0, \ldots , n-1 $, $ \lambda >  0 $ and $ \mathcal{C}_m(C,\cdot) = \lambda \mathcal{C}_n(C, \cdot) $, then $ C $ is a round ball.
	\end{Theorem}
	\noindent Related characterizations of the round ball among arbitrary convex bodies can be found in \cite{MR1310959}, \cite{MR1604003} and \cite{santilli2020uniqueness}. A very detailed account on many other uniqueness results for convex bodies can be found in \cite{MR3155183}.
	
	\paragraph{} Besides of being a central result in the theory of convex bodies, Theorem \ref{theo: schneider} has recently emerged in \cite{MR4214340} as an important tool to study the asymptotic behaviour of mixed-volume preserving flows in $ \mathbf{R}^{n+1} $. If $ \Omega\subseteq \mathbf{R}^{n+1} $ is a compact domain with $ \mathcal{C}^2 $-boundary, we define for $ k= 0, \ldots , n $ the $ n+1-k $ mixed volume $ V_{n+1-k}(\Omega) $ as
	\begin{equation*}
		V_{n+1-k}(\Omega) := \int_{\partial \Omega_t} E_{k-1}(x)\, d\mathcal{H}^n(x),
	\end{equation*}
	where $V_{n+1}(\Omega) = (n+1)\mathcal{L}^{n+1}(\Omega) $ and $E_k $ is the normalized $ k $-th mean curvature of $ \partial \Omega $, namely $ E_k = {n \choose k}^{-1}H_k $.  Consider a smooth embedding $ X_0 : M \rightarrow \mathbf{R}^{n+1} $ of a closed $ n $-dimensional manifold $ M $ such that $ X_0(M) = \partial \Omega_0 $ is a smooth strictly convex hypersurface and a smooth flow $ X: M^n \times [0, T) \rightarrow \mathbf{R}^{n+1} $ of the form
	\begin{equation}\label{eq: isotropic flow}
		\begin{cases}
			\frac{\partial X}{\partial t}(x,t) = (\mu(t) - E_k(x,t)^{\alpha/k})\nu(x,t)\\
			X(\cdot, 0) = X_0,
		\end{cases}
	\end{equation}
	where $ \alpha > 0 $, $ \nu(\cdot,t) $ is the outward unit-normal of the hypersurface $ X(M,t)= \partial \Omega_t $, $ k \in \{1, \ldots , n\} $ and $ \mu(t) $ is chosen to keep constant a general monotone function of $ V_{n+1-k}(\Omega_t) $ and $ V_{n+1}(\Omega_t) $ along the flow (in particular allowing to keep constant along the flow either only $ V_{n+1-k}(\Omega_t) $ or only $ V_{n+1}(\Omega_t) $); see \cite[pag.\ 194]{MR4214340}. The following bubbling theorem for the geometric flow in \eqref{eq: isotropic flow} is proved in \cite{MR4214340}.
	\begin{Theorem}[\protect{cf.\ \cite[Theorem 1.1]{MR4214340}}]\label{theo: andrews-wei}
		The flow in \eqref{eq: isotropic flow} has a smooth strictly convex solution $ \partial \Omega_t $ defined for all $ t \geq 0 $ and $ \partial \Omega_t $ smoothly converges to a round sphere as $ t \to \infty $.
	\end{Theorem}
	\noindent Special cases of this result were known before; see \cite[pp.\ 195-196]{MR4214340} for a detailed account. In particular, Sinestrari in \cite{MR3396440} treats the case $ k = 1 $ and $ \alpha > 0 $ (see also \cite{MR3829571}). The main novelty of Theorem \ref{theo: andrews-wei} is the treatment of the case $ k > 1 $ (for any exponent $ \alpha > 0 $). As explained in \cite{MR4214340}, the asymptotic analysis for the case $ k > 1 $ cannot be done in this case with the previously known approaches and the authors introduce a new method based on Theorem \ref{theo: schneider}.
	
	\paragraph{} It is a natural question to extend Theorem \ref{theo: andrews-wei} to the anisotropic setting and in \cite{MR4214340} the authors make a conjecture in this direction. Before introducing the problem, let us briefly review few definitions. Let $ \phi $ be a uniformly convex smooth norm on $ \mathbf{R}^{n+1} $ and we denote with $ \phi^\ast $ its dual norm. The set $ \mathcal{W}^\phi = \{x \in \mathbf{R}^{n+1}: \phi^\ast(x)=1\} $ is called Wulff shape of $ \phi $. If $ M = \partial \Omega \subseteq \mathbf{R}^{n+1} $ is a $ \mathcal{C}^2 $-hypersurface and $ \eta : M \rightarrow \mathbf{S}^n $ is the outward unit-normal vector field, then we define the (outward) anisotropic normal $ \nu : M \rightarrow \mathcal{W}^\phi $ setting $ \nu(x) = \nabla \phi(\eta(x)) $ for $ x\in M $. One can prove that $ \Der \nu(x) \in \Hom(\Tan(M,x), \Tan(M,x)) $ and $ \Der \nu(x) $ is diagonalizable with real eigenvalues $ \kappa^\phi_1(x) \leq \ldots \leq \kappa^\phi_n(x) $. We refer to these eigenvalues as the anisotropic principal curvatures of $ M $ at $ x $. Then we define the anisotropic $k $-th man curvature of $ M $ as
	\begin{equation*}
		H^\phi_{k}(x) = \sum_{1 \leq l_1 < \ldots < l_k \leq n}\kappa^\phi_{l_1}(x)\cdots \kappa^\phi_{l_k}(x) \qquad \textrm{for $ k =0, \ldots , n $},
	\end{equation*}
	where $ H^\phi_0 \equiv 1 $. We define  the renormalized $ k $-th mean curvature $ E^\phi_k = {n \choose k}^{-1} H^\phi_k $ and the anisotropic mixed-volume
	\begin{equation*}
		V_{n+1-k}^\phi(\Omega) = \int_{\partial \Omega} \phi(\eta(x))E^\phi_{k-1}(x)\, d\mathcal{H}^n(x) \qquad \textrm{for $ k =0, \ldots , n $,}
	\end{equation*}
	where $ \eta $ is the outward unit-normal of $ \Omega $ and $ V^\phi_{n+1}(\Omega) = (n+1)\mathcal{L}^{n+1}(\Omega) $. Using these definitions we can naturally formulate a general anisotropic version  of \eqref{eq: isotropic flow}. Given a closed strictly convex hypersurface $ X_0 : M \rightarrow \mathbf{R}^{n+1} $, $ \alpha > 0 $, $ k \in \{1, \ldots , n\} $, we consider the flow of the form
	\begin{equation}\label{eq: anisotropic flow}
		\begin{cases}
			\frac{\partial X}{\partial t}(x,t) = (\mu^\phi(t) - E^\phi_k(x,t)^{\alpha/k})\nu^\phi(x,t)\\
			X(\cdot, 0) = X_0,
		\end{cases}
	\end{equation}
	where $ \nu^\phi(\cdot, t) $ is the outward anisotropic normal of $ X(M,t) = \partial \Omega_t $ and the term $ \mu^\phi(t) $ is chosen to keep constant along the flow a monotone function $ V^\phi_{n+1-k}(\Omega_t) $ and $ V_{n+1}^\phi(\Omega_t] $. It is remarked in \cite[pag.\ 219]{MR4214340} that \emph{if} Theorem \ref{theo: schneider} can be extended to the anisotropic setting, \emph{then} the method of \cite{MR4214340} carries through directly to prove the following Conjecture on the asymptotic behaviour of the anisotropic flow in \eqref{eq: anisotropic flow}.
	\begin{Conjecture}[\protect{cf.\ \cite{MR4214340}}]\label{conj: anisotropic flow}
		Suppose $ k \in \{2, \ldots , n-1\} $.	The flow in \eqref{eq: isotropic flow} has a smooth strictly convex solution $ \partial \Omega_t $ defined for all $ t \geq 0 $ and $ \partial \Omega_t $ smoothly converges to a scaled and translated Wulff shape of $ \phi $ as $ t \to \infty $.
	\end{Conjecture}
	\noindent We remark that the cases $ k =1 $ and $ k = n $ can be proved with other techniques (see \cite[pag.\ 219]{MR4214340}). 
	
	To formulate the anisotropic version of Theorem \ref{theo: schneider}, we first need to introduce the anisotropic analogous of curvature measures. To the best of our knowledge this notion has been first studied in \cite{MR1782274}. If $ C \subseteq \mathbf{R}^{n+1} $ is an arbitrary convex body, we denote with $ \bm{\delta}^\phi_C$ the distance function from $ C $ with respect to the norm $ \phi $ and with $ \bm{\xi}^\phi_C $ the metric projection onto $ C $ with respect to $ \phi $. Then there exist uniquely determined Radon measures $ \mathcal{C}^\phi_0(C, \cdot), \ldots , \mathcal{C}^\phi_n(C, \cdot) $ supported on the boundary $ \partial C $ of $ C $ such that 
	\begin{equation*}
		\mathcal{L}^{n+1}(\{ x \in \mathbf{R}^{n+1}: 0 < \bm{\delta}^\phi_C(x)\leq \rho, \; \bm{\xi}^\phi_C(x) \in B  \}) = \sum_{m=0}^{n} \rho^{n+1-m}\mathcal{C}^\phi_m(C, B)
	\end{equation*}
	for every Borel subset $ B \subseteq \partial C $; cf.\ \cite[Theorem 2.3]{MR1782274}. The following conjecture is explicitly formulated in \cite{MR4214340}.
	
	\begin{Conjecture}[\protect{cf.\ \cite[8.2]{MR4214340}}]\label{conj: anisotropic schneider}
		If $ C \subseteq \mathbf{R}^{n+1} $ is an arbitrary convex body, $ m = 0, \ldots , n-1 $, $ \lambda >  0 $ and $ \mathcal{C}^\phi_m(C,\cdot) = \lambda \mathcal{C}^\phi_n(C, \cdot) $, then $ C $ is a scaled and translated Wulff shape of $ \phi $.
	\end{Conjecture}
	\noindent In this paper we prove Conjecture \ref{conj: anisotropic schneider}. As already pointed out, Conjecture \ref{conj: anisotropic flow} follows from Conjecture \ref{conj: anisotropic schneider} with a direct extension  of the method in \cite{MR4214340} to the anisotropic setting.
	
	\paragraph{} We now describe the content of the paper. In the preliminary section \ref{Section: preliminaries}, besides introducing the notation and recalling few facts on the geometry of the Wulff shapes, we introduce a notion of anisotropic normal bundle for closed sets and we recall from the recent work \cite{kolasinski2021regularity} the Lipschitz and differentiability properties of the anisotropic nearest point projection onto a closed set. The nearest point projection onto a closed set $ K $ might not be single-valued at some points of $ \mathbf{R}^{n+1}\setminus K $ (if it is everywhere single-valued then the set $ K $ is convex); indeed the set of points where it fails to be single-valued might even be dense in $ \mathbf{R}^{n+1} \setminus K $, and  even if $ K $ is the complementary of a convex body with $ \mathcal{C}^1 $-boundary; see \cite{MR4279967}. Therefore in our proofs we use an approach based on the theory of multivalued functions.
	
	In Section \ref{Section: closed} we introduce the anisotropic higher-order mean curvatures for an arbitrary closed set  $ K $ as functions on the anisotropic normal bundle. We use them to find a (local) formula for the anisotropic tubular neighbourhood around $ K $ in Theorem \ref{theo: Steiner closed}. This formula generalizes the isotropic Steiner formula for arbitrary closed sets proved in \cite{MR2031455}. Our motivation to develop a theory for arbitrary (non-convex) closed sets is given by the fact that in our proof of Theorem \ref{theo: heintze karcher} we need to work with the complementary of the convex body $ C $.  In fact, the proof of the inequality employs the general Steiner formula of Theorem \ref{theo: Steiner closed} with $ K = \mathbf{R}^{n+1}\setminus C $, while the analysis of the equality case uses Corollary \ref{theo: positive reach} of Theorem \ref{theo: Steiner closed} always for the complementary of $ C $.
	
	In Section \ref{Section: convex}, specializing the Steiner formula of section \ref{Section: closed} to convex sets, we obtain an integral representation of the anisotropic curvature measures. This result extends to the anisotropic setting the classical integral representation in \cite{MR849863} for the isotropic curvature measures. Then we prove an anisotropic version of the Minkowski formulae for arbitrary convex bodies and we use them to study the $ k $-th mean curvature of an arbitrary convex body $ C \subseteq \mathbf{R}^{n+1} $, whose $(n-k)$-th anisotropic curvature measure satisfies $ C^\phi_{n-k}(K, \cdot) = \lambda C^\phi_n(K, \cdot) $ for some constant $ \lambda > 0 $. In particular we obtain a lower bound for $ \lambda $ in terms of the anisotropic isoperimetric ratio of $ C $. 
	
	Finally in Section \ref{Section: heintze karcher} we prove an optimal geometric inequality for arbitrary convex bodies. This inequality is inspired by an inequality originally found by Heintze-Karcher in \cite{MR533065}  and used by Montiel-Ros in \cite{MR1173047} to prove the uniqueness of compact smooth hypersurfaces with constant higher-order mean curvature.  Combining the optimal geometric inequality with the lower bound for the constant $ \lambda $ in section \ref{Section: convex} we finally obtain the proof of conjecture \ref{conj: anisotropic schneider}.
	
	\section{Preliminaries}\label{Section: preliminaries}
	In general, but with few exceptions explained below, we follow the notation and terminology of \cite{MR0257325} (see \cite[pp. 669-676]{MR0257325}). 
	
	We denote by $ \bullet $ a fixed scalar product on $ \mathbf{R}^{n+1} $ and by $ | \cdot | $ its associated norm. We denote with $ \mathbf{S}^n $ the unit sphere; i.e.\ $ \mathbf{S}^n = \{x \in \mathbf{R}^{n+1}: |x|=1\} $. The map $\mathbf{p} : \mathbf{R}^{n+1} \times \mathbf{R}^{n+1} \rightarrow \mathbf{R}^{n+1} $ is the projection onto the first component; i.e.\ $\mathbf{p}(x, \eta) = x$. If $ S \subseteq \mathbf{R}^k $ and $ a \in \mathbf{R}^p $, then we denote with $ \Tan(S,a) $ the tangent cone of $ S $ at $ a $ (see \cite[3.1.21]{MR0257325}) and with $ \Tan^m(\mathcal{H}^m\restrict S, a) $ the cone of all $(\mathcal{H}^m\restrict S, m) $ approximate tangent vectors at $ a $ (see \cite[3.2.16]{MR0257325}). For an $ (\mathcal{H}^m, m) $ rectifiable and $ \mathcal{H}^m $-measurable set $ S \subseteq \mathbf{R}^p $, the cone $ \Tan^m(\mathcal{H}^m\restrict S, a) $ is an $ m $-dimensional plane for $ \mathcal{H}^m $ a.e.\ $a \in S $ (see \cite[3.2.19]{MR0257325}); in this case, if $ f : S \rightarrow \mathbf{R}^q$ is a Lipschitzian function and $ k \in \{1, \ldots , m\}$, then we denote with $ \ap J^S_k f $ its $(\mathcal{H}^m\restrict S, m)$-approximate $ k $-dimensional Jacobian. See \cite[3.2.19, 3.2.10 and 3.2.22]{MR0257325} for details about this definition and applications to area/coarea formula that will be used in this paper.

	If $ X $ is a topological space and $ S \subseteq X $ then its topological boundary is $ \partial S $ and the characteristic map of $ S $ is $ \bm{1}_S $.
	If $ Q \subseteq X \times Y $ and $ S \subseteq X $, we define $ Q|S = \{ (x, y) \in Q: x \in S  \} $.
	
	\subsection{Multivalued maps}
	
	A map $ T $ defined on a set $ X $ is called \emph{$Y$-multivalued}, if $ T(x) $ is a subset of $ Y $ for every $ x \in X $. If $ T(x) $ is a singleton, with a little abuse of notation we denote with $ T(x) $ the unique element of the set $ T(x) \subseteq Y $. Suppose $(X, \|\cdot \|)$ and $(Y, \|\cdot \|)$ are finite dimensional normed
	vectorspaces, $T$ is a~$Y$-multivalued map such that $ T(u) \neq \varnothing $ for every $ u \in X $ and $ x \in X $.
	\begin{enumerate}
		\item  We say that $ T $ is \emph{weakly continuous at $ x $} if and only if for every $ \epsilon > 0 $ there exists $ \delta > 0 $ such that 
		\begin{equation*}
			T(y) \subseteq T(x) + \{v \in Y : \|v\| < \epsilon\} \qquad \textrm{whenever $ \|y-x\| < \delta $;}
		\end{equation*}
		if, additionally,  $ T(x)$ is a singleton, then we say that $ T $ is continuous at $ x $.
		\item  We~say that \emph{$T$ is strongly
			differentiable at $x \in X$} if and only if $T(x)$ is a singleton and there exists a linear map $L : X \to Y$ such that for any $\varepsilon > 0$
		there exists $\delta > 0$ satisfying
		\begin{displaymath}
			\| w - T(x) - L(y-x)\| \leq \varepsilon \|y-x\|
			\quad \text{whenever $\|x-y\| \le \delta$ and $w \in T(y)$}; 
		\end{displaymath} 
		cf.\ \cite[Definition 2.21]{kolasinski2021regularity}. The linear map $ L $ is unique (cf.\ \cite[Remark 2.22]{kolasinski2021regularity}) and we denote it with $ \Der T(x) $. Moreover we denote with $ \dmn \Der T $ the set of points $ x \in X $ where $ T $ is strongly differentiable.
	\end{enumerate}
	
	\noindent The following general fact on the Borel measurability of the differential of a multivalued map will be useful.
	
	\begin{Lemma}\label{lem : Borel measurability differential}
		Suppose $(X, \|\cdot \|)$ and $(Y, \|\cdot \|)$ are finite dimensional normed
		vectorspaces, $T$ is a~$Y$-multivalued weakly continuous map such that $ T(u) \neq \varnothing $ for every $ u \in X $. 
		
		Then $\{x \in X : \textrm{$T(x)$ is a singleton}   \} $ and $ \dmn \Der T $ are Borel subsets of $ X $ and $ \Der T : \dmn \Der T \rightarrow \Hom(X,Y) $ is a Borel function.
	\end{Lemma}
	
	\begin{proof}
		We define $ U = \{x \in X: \textrm{$T(x)$ is a singleton}\} $ and the function $ \diam : \bm{2}^Y \setminus \{\varnothing\} \rightarrow [0, \infty] $ as $ \diam S = \sup \{\|y_1 - y_2 \| : y_1, y_2 \in S\} $ for every $ S \in \bm{2}^Y \setminus \{\varnothing\} $. Noting that $ \diam \circ T : X \rightarrow [0, +\infty] $ is upper-semicontinuous, we conclude that $ U = \{x \in X : \diam(T(x))=0    \} $ is a Borel subset of $ X $.
		
		For the positive integers $ i, j $ we define 
		\begin{equation*}
			C_{ij} = \bigg\{ (x, L) \in U \times \Hom(X,Y) : \|w - T(x) - L(h)\| \leq \frac{1}{i}\| h \| \quad \textrm{for $\|h\| \leq \frac{1}{j} $ and $ w \in T(x+h) $} \bigg\}.
		\end{equation*}
		We prove that $ C_{ij} $ is relatively closed in $ U \times \Hom(X,Y) $. By contradiction assume that $C_{ij} $ is not closed. Then there exists $ (x_0, L_0)\in (U \times \Hom(X,Y)) \setminus C_{ij} $ and  a sequence $(x_k, L_k) \in C_{ij} $ converging to $(x_0, L_0) $. We notice that there exist $ h_0 \in X $ with $ \|h_0 \| \leq \frac{1}{j} $ and $ w_0 \in T(x_0 + h_0) $ such that
		\begin{equation*}
			\| w_0- T(x_0) - L_0(h_0)\| > \frac{1}{i}\|h_0 \|,
		\end{equation*}
		we define $ h_k = x_0 + h_0 - x_k $ for every $k\geq1 $ and we select $ k_0\geq 1 $ so that $ \|h_k \| \leq \frac{1}{j} $ for every $ k \geq k_0 $. Since $ x_0 + h_0 = x_k + h_k $ and $ w_0 \in T(x_k+ h_k) $ for every $ k \geq 1 $, we infer that
		\begin{equation*}
			\|w_0 - T(x_k)-L_k(h_k)\| \leq \frac{1}{i}\|h_k \| \qquad \textrm{for every $ k \geq k_0 $}. 
		\end{equation*}
		Noting that $ T(x_k) \to T(x_0)$ and $ h_k \to h_0 $ as $ k \to \infty $, we deduce that 
		\begin{equation*}
			\|w_0 - T(x_0)-L_0(h_0)\| \leq \frac{1}{i}\|h_0 \|
		\end{equation*}
		and we obtain a contradiction.
		
		Let $G : =\{ (x, \Der T(x)) : x \in \dmn \Der T  \}$ and $ \pi_X : X \times \Hom(X,Y) \rightarrow X $, $ \pi_X(x, T) = x $ for every $(x,T)\in X \times \Hom(X,Y) $. Noting that
		\begin{equation*}
			G = \bigcap_{i=1}^\infty \bigcup_{j=1}^\infty C_{ij},
		\end{equation*}
		we infer that $ G $ is a Borel subset of $ U \times \Hom(X,Y) $. Since $ \pi|G $ is injective, we obtain the conclusion from \cite[2.2.10]{MR0257325}.
	\end{proof}

	\subsection{Norms and Wulff shapes.} 
	
	Let $ \phi $ be a norm on $ \mathbf{R}^{n+1} $. We say that $ \phi $ is a \emph{$ \mathcal{C}^k $-norm} if and only if $ \phi \in \mathcal{C}^k(\mathbf{R}^{n+1} \setminus \{0\}) $. We say that $ \phi $ is \emph{uniformly convex} if and only if there exists a constant $ \gamma > 0 $ (ellipticity constant) such that the function $ \mathbf{R}^{n+1} \ni u \rightarrow \phi(u) - \gamma |u| $ is convex. If $ \phi $ is a uniformly convex $ \mathcal{C}^2 $-norm then 
	\begin{equation*}
		\Der^2\phi(u)(v,v) \geq \gamma |v|^2 
	\end{equation*}
	for all $ u \in \mathbf{R}^{n+1}$ with $|u|=1 $ and for all $ v \in \mathbf{R}^{n+1} $ perpendicular to $ u $. 
	
	For any norm $ \phi $ we denote by $ \phi^\ast $ \emph{the conjugate norm} of $ \phi $; namely if $ u \in \mathbf{R}^n $ then  $ \phi^\ast(u) = \sup \{ v \bullet u : \phi(v) = 1    \} $. It is well known that if $ \phi $ is a uniformly convex $ \mathcal{C}^2 $-norm then $ \phi^\ast $ is a uniformly convex $ \mathcal{C}^2 $-norm. We refer to  \cite[Lemma 2.32]{MR4160798} for this and other basic facts on $ \phi $ and $ \phi^\ast $. These facts will be tacitly used through the paper. If $ B = \{x \in \mathbf{R}^{n+1} : \phi^\ast(x)\leq 1\} $ we define the \emph{Wulff shape} (or \emph{Wulff crystal}) of $ \phi $ as
	\begin{equation*}
		\mathcal{W}^\phi = \partial B.
	\end{equation*}
	If $ \phi $ is a uniformly convex $ \mathcal{C}^2 $-norm then the Wulff shape of $ \phi $ is a uniformly convex $ \mathcal{C}^2 $-hypersurface. In this case we denote the exterior unit-normal of $ B $ with $ \bm{n}^\phi : \mathcal{W}^\phi \rightarrow \mathbf{S}^n $;  we remark (see \cite[2.32]{MR4160798}) that $ \bm{n}^\phi $ is a $ \mathcal{C}^1 $-diffeomorphism onto $ \mathbf{S}^n $ and 
	\begin{equation}\label{eq: normal wulff shape and gradient phi}
		\nabla \phi(\bm{n}^\phi(u)) = u \quad \textrm{for $ u \in \mathcal{W}^\phi $,} \qquad \bm{n}^\phi(\nabla \phi(\eta)) = \eta \quad \textrm{for $ \eta \in \mathbf{S}^n $.}
	\end{equation}

	\subsection{Distance function and normal bundle.} 
	\paragraph{Warning.} In this paper sometimes we refer to \cite{kolasinski2021regularity}. Notice that in this paper we use the same symbols with a different meaning; compare the definitions below with those given in \cite[Section 2]{kolasinski2021regularity}.
	
	Suppose $ K\subseteq \mathbf{R}^{n+1} $ is closed and $ \phi $ is a uniformly convex $ \mathcal{C}^2 $-norm on $ \mathbf{R}^{n+1} $. If $ \phi $ is the Euclidean norm the dependence on $ \phi $ is omitted in all the symbols introduced below.  
	
	The \emph{$\phi $-distance function}  $ \bm{\bm{\delta}}^\phi_K : \mathbf{R}^{n+1} \rightarrow \mathbf{R} $ is defined by
	\begin{displaymath}
		\bm{\bm{\delta}}^\phi_K(x) = \inf\{ \phi^\ast(x-c) : c \in K  \} \quad \textrm{for every $ x \in \mathbf{R}^{n+1} $}
	\end{displaymath}
	and 
	\begin{equation*}
		S^\phi(K,r) = \{x : \bm{\delta}^\phi_K(x) = r   \} \qquad \textrm{for $ r > 0 $.}
	\end{equation*}
	The set $\Unp^\phi(K) $ is the set of $ x \in \mathbf{R}^{n+1} \setminus K $ such that there exists a unique $ c \in K $ with $ \phi^\ast(x-c) = \bm{\delta}^\phi_K(x) $.  Since $ \bm{\delta}^\phi_K $ is a Lipschitz map, it follows that $ \mathcal{L}^{n+1}(\mathbf{R}^{n+1} \setminus (K \cup \Unp^\phi(K))) =0 $. The \emph{nearest $ \phi $-projection} $ \bm{\xi}^\phi_K : \mathbf{R}^{n+1}  \rightarrow \bm{2}^{K} $ is the $K$-multivalued map characterized by
	\begin{displaymath}
		\bm{\xi}^\phi_K(x) = \{c \in K: \bm{\delta}^\phi_K(x) = \phi^\ast(x-c)\}  \qquad \textrm{for every $ x\in \mathbf{R}^{n+1} $.}
	\end{displaymath}
	This is a weakly continuous by \cite[Lemma 2.30(f)]{kolasinski2021regularity}; moreover notice that $ \Unp^\phi(K) $ is a Borel subset of $ \mathbf{R}^{n+1} $ by Lemma \ref{lem : Borel measurability differential}. The  $ \phi $-Cahn-Hoffman map of $ K $ is the $ \mathcal{W}^\phi $-multivalued function $ \bm{\nu}^\phi_K: \mathbf{R}^{n+1} \setminus K \rightarrow \bm{2}^{\mathcal{W}^\phi} $ defined by
	\begin{equation*}
		\bm{\nu}^\phi_K(x) =\bm{\delta}_K^\phi(x)^{-1} (x - \bm{\xi}_K^\phi(x)) \qquad \textrm{for $ x \in \mathbf{R}^{n+1} \setminus K $.}
	\end{equation*}
	Finally we set $ \bm{\psi}_K^\phi : \mathbf{R}^{n+1} \setminus K \rightarrow \mathbf{2}^K \times \bm{2}^{\mathcal{W}^\phi} $ by
	\begin{equation*}
		\bm{\psi}_K^\phi(x) = (\bm{\xi}_K^\phi(x), \bm{\nu}_K^\phi(x)) \qquad \textrm{for $ x \in \mathbf{R}^{n+1} \setminus K $.}
	\end{equation*}
	
	We define \emph{the $ \phi $-unit normal bundle of $ K $} as
	\begin{equation*}
		N^\phi(K) = \{ (x, \eta ) : x \in K, \;  \eta \in \mathcal{W}^\phi, \; \bm{\delta}_K^\phi(x + r\eta) = r\; \textrm{for some $ r > 0 $}  \}
	\end{equation*}
	and we set $ N^\phi(K,x) = N^\phi(K)|\{x\} $. We recall (cf.\ \cite[Lemma 5.2]{MR4160798}) that $ N^{\phi}(K) $ is Borel and countably
	$(n-1)$-rectifiable (in the sense of~\cite[3.2.14(2)]{MR0257325}) subset of $ \mathbf{R}^{n+1} \times \mathbf{R}^{n+1} $; moreover 
	\begin{equation}\label{eq: phi normal vs euclidean normal}
		N^\phi(K) = \{(a, \nabla \phi(\eta)) : (a, \eta)\in N(K)\}.
	\end{equation}

	The \emph{normal $\phi$-distance function to the
		cut locus of $ K $} is the upper-semicontinuous function $\bm{r}_{K}^{\phi}: N^{\phi}(K) \rightarrow (0, +
	\infty]$ given~by
	\begin{equation*}
		\bm{r}_{K}^{\phi}(a,\eta) = \sup \bigl\{ s : \bm{\delta}^\phi_K(a + s\eta) = s \bigr\}
		\qquad \text{for $(a,\eta)\in N^{\phi}(K) $}
	\end{equation*}
	and \emph{the $ \phi $-cut locus of $ K $} is given by
	\begin{equation*}
		\Cut^{\phi}(K) = \bigl\{ a + \bm{r}^\phi_K(a, \eta)\eta : (a, \eta) \in
		N^{\phi}(K) \bigr\}.
	\end{equation*}
	We recall that $ \mathcal{L}^{n+1}(\Cut^\phi(K)) =0 $; cf.\ \cite[Remark 5.11]{MR4160798}; if $ K $ is convex then $ \Cut^\phi(K) = \varnothing $. A related function which will be useful in the sequel is defined as
	\begin{equation*}
		\bm{\rho}^\phi_K(x) = \sup\{ s \geq 0 :  \bm{\delta}^\phi_K(a + s (x - a)) = s \bm{\delta}^\phi_K(x)  \} \qquad \textrm{for $ x \in \mathbf{R}^{n+1} \setminus K $ and $ a \in \bm{\xi}^\phi_K(x) $.}
	\end{equation*}
	This definition does not depend on the choice of $ a  \in \bm{\xi}^\phi_K(x) $ and the function $\bm{\rho}^\phi_K : \mathbf{R}^{n+1} \setminus K \rightarrow [1, +\infty] $ is upper-semicontinuous; cf.\ \cite[Remark 2.29]{kolasinski2021regularity}. Notice that $  \{ x : \bm{\rho}^\phi_K(x)> 1  \} \subseteq \Unp^\phi(K) $ and 
	\begin{equation}\label{eq : r and rho}
		\bm{r}^\phi_K(a,u) = r \bm{\rho}^\phi_K(a+ru) \qquad \textrm{for every $ (a,u) \in N^\phi(K) $ and $ 0 < r < \bm{r}^\phi_K(a,u) $}
	\end{equation}

	Let $ \Sigma^\phi(K) $ be the set of non-differentiability points of $ \bm{\delta}^\phi_K $ in $ \mathbf{R}^{n+1} \setminus K $. It is well known (see \cite{kolasinski2021regularity} and references therein) that
	\begin{equation*}
		\Sigma^{\phi}(K) \subseteq \Cut^{\phi}(K) \subseteq \Clos{\Sigma^{\phi}(K)} \qquad \textrm{and} \qquad \Sigma^\phi(K) = \mathbf{R}^{n+1} \setminus(K \cup \Unp^\phi(K));
	\end{equation*}
	moreover for $ x \in \Unp^\phi(K) $  
	\begin{equation}\label{eq: gradient and normal}
		\nabla \bm{\delta}^\phi_K(x) = \nabla \phi^\ast(x - \bm{\xi}^\phi_K(x)) \in \mathcal{W}^{\phi^\ast}\qquad \textrm{and} \qquad   \nabla \phi(\nabla \bm{\delta}^\phi_K(x)) =  \bm{\nu}^\phi_K(x)\in \mathcal{W}^\phi
	\end{equation}
	cf.\ \cite[Lemma 2.30(c)]{kolasinski2021regularity}. It follows from \cite[Lemma 2.32]{MR4160798} that $\bm{\nu}^\phi_K(x) \bullet \nabla \bm{\delta}^\phi_K(x) = \phi(\nabla \bm{\delta}^\phi_K(x)) = 1 $ for $ x \in \Unp^\phi(K) $; in particular $\bm{\nu}^\phi_K(x) $ and $\nabla \bm{\delta}^\phi_K(x)$ are linearly independent.

	\begin{Theorem}[\protect{cf.\ \cite[Corollary 3.10]{kolasinski2021regularity}}]
		\label{theo: projectionLipschitz}
		Suppose $\phi $ is a uniformly convex $ \mathcal{C}^2 $-norm on $ \mathbf{R}^{n+1} $, $K \subseteq \mathbf{R}^{n+1}$ is closed,
		$1 < \lambda < \infty$, $0 < s < t < \infty$, and
		\begin{displaymath}
			K_{\lambda,s,t} = \bigl\{ x \in \mathbf{R}^{n+1} \setminus K: \bm{\rho}^\phi_K(x)\geq \lambda, \; s \leq \bm{\delta}^\phi_K(x) \leq t \bigr\} \,.
		\end{displaymath} 
		Then $ \bm{\xi}_{K}^{\phi} | K_{\lambda,s,t} $ is Lipschitz continuous.
	\end{Theorem}

	\begin{Theorem}[\protect{cf.\ \cite[Theorem 1.4]{kolasinski2021regularity}}]\label{theo: distance twice diff}
		Suppose $\phi $ is a uniformly convex $ \mathcal{C}^2 $-norm on $ \mathbf{R}^{n+1} $ and $K \subseteq \mathbf{R}^{n+1}$ is closed.
		
		Then $ \mathcal{L}^{n+1}(\mathbf{R}^{n+1} \setminus (K \cup \dmn(\Der \bm{\nu}^\phi_K))) =0 $ and
		\begin{equation*}
			\{ a + r \eta: 0 < r < \bm{r}_K^{\phi}(a, \eta) \} \subseteq \dmn(\Der \bm{\nu}^\phi_K)
		\end{equation*}  
		for $\mathcal{H}^{n}$ almost all $(a, \eta) \in N^\phi(K) $.
	\end{Theorem}

	\section{Anisotropic Steiner formula for arbitrary closed sets}\label{Section: closed}

	In this section we assume  that $ \phi $ is a uniformly convex $ \mathcal{C}^2 $ norm. We start recalling few known facts.
	
	\begin{Remark}\label{rem: tangent ot level sets}
		Suppose $ K \subseteq \mathbf{R}^{n+1} $ is closed.	If $ x \in \Unp^\phi(K) $, $ r = \bm{\delta}^\phi_K(x) $, $ 0 < t < 1 $ and $ y = \bm{\xi}^\phi_K(x) + tr \bm{\nu}^\phi_K(x) $, then $ y \in \Unp^\phi(K) $ and
		\begin{equation*}
			\Tan(S^\phi(K,r) , x) = \{v \in \mathbf{R}^{n+1} : v \bullet \nabla \bm{\delta}^\phi_K(x) =0\},
		\end{equation*}
		\begin{equation*}
			\nabla \bm{\delta}^\phi_K(x) = \nabla \bm{\delta}^\phi_K(y), \qquad \Tan(S^\phi(K,r) , x) = \Tan(S^\phi(K,tr) , y). 
		\end{equation*}
		cf.\ \cite[Lemma 2.30, Lemma 2.33]{kolasinski2021regularity}.
	\end{Remark}
	
	\begin{Remark}\label{rem: tangent level sets and Wulff shapes}
		For every $(a,u)\in N^\phi(K) $ and for every $ 0 < r < \bm{r}^\phi_K(a,u) $ we have that
		\begin{equation*}
			\Tan(S^\phi(K,r), a + ru) = \Tan(\mathcal{W}^\phi, u). 
		\end{equation*}
		Setting $ \eta = \frac{\nabla \bm{\delta}^\phi_K(a+ ru)}{|\nabla\bm{\delta}^\phi_K(a+ ru)|} $, this assertion follows from Remark \ref{rem: tangent ot level sets}, noting that (see \eqref{eq: normal wulff shape and gradient phi} and \eqref{eq: gradient and normal})
		\begin{equation*}
			\nabla \phi(\eta) = \nabla \phi(\nabla\bm{\delta}^\phi_K(a+ ru)) = u, \qquad \bm{n}^\phi(u) = \eta.
		\end{equation*}
	\end{Remark}

	\begin{Lemma}[\protect{cf.\ \cite[Lemma 4.8]{kolasinski2021regularity}}]\label{lem: chi curvatures}
		Suppose $ K \subseteq \mathbf{R}^{n+1} $ is closed, $ x \in \dmn (\Der \bm{\nu}^\phi_K) $ and $ r = \bm{\delta}^\phi_K(x) $. 
		
		Then $ \bm{\rho}^\phi_K(x) > 1 $, $ \Der \bm{\nu}^\phi_K(x)(\bm{\nu}^\phi_K(x)) =0 $, $\Der \bm{\nu}^\phi_K(x)[\Tan(S^\phi(K,r), x)] \subseteq \Tan(S^\phi(K,r), x)  $ and the linear automorphism  $ \Der \bm{\nu}^\phi_K(x)|\Tan(S^\phi(K,r), x) $ is diagonalizable with real eigenvalues. Moreover for every eigenvalue $ \rchi $ of $ \Der \bm{\nu}^\phi_K(x)|\Tan(S^\phi(K,r), x) $ holds that
		\begin{equation*}
			\frac{1}{(1-\bm{\rho}^\phi_K(x)) r} \leq \chi\leq \frac{1}{r}.
		\end{equation*}
	\end{Lemma}
	
	\begin{proof}
		All the assertions of this statement with the exception of the diagonalizability property of $ \Der \bm{\nu}^\phi_K(x)|\Tan(S^\phi(K,r), x) $, are contained in  \cite[Lemma 4.8]{kolasinski2021regularity}. Indeed, the first part of the proof of this lemma in combination with \cite[Remark 2.25]{MR4160798} shows that $ \Der \bm{\nu}^\phi_K(x)|\Tan(S^\phi(K,r), x) $ is diagonalizable. 
	\end{proof}

	\begin{Definition}
		Suppose  $ K \subseteq \mathbf{R}^{n+1} $ is closed, $ x \in \dmn (\Der \bm{\nu}^\phi_K) $  and $ r = \bm{\delta}^\phi_K(x) $. We define
		\begin{equation*}
			\rchi^\phi_{K,1}(x) \leq \ldots \leq \rchi_{K,n}^\phi(x)
		\end{equation*}
		to be the eigenvalues (counted with their algebraic multiplicity) of $ \Der \bm{\nu}^\phi_K(x)| \Tan(S^\phi(K,r), x) $.
	\end{Definition}
	
	\begin{Lemma}\label{rem: borel meas of chi}
		The set $ \dmn \Der \bm{\nu}^\phi_K $ is a Borel subset of $ \mathbf{R}^{n+1} $ and the functions $ \rchi^\phi_{K,i} : \dmn \Der \bm{\nu}^\phi_K \rightarrow \mathbf{R} $ are Borel functions.
	\end{Lemma}
	
	\begin{proof}
		Let $ X $ be the set of all authomorphisms $ T \in \Hom(\mathbf{R}^{n+1}, \mathbf{R}^{n+1}) $ with real eigenvalues. For each $ T \in X $ we define $ \lambda_0(T) \leq \ldots \leq \lambda_n(T) $ to be the eigenvalues of $ T $ counted with their algebraic multiplicity and we define the map $ \lambda : X \rightarrow \mathbf{R}^{n+1} $ as 
		\begin{equation*}
			\lambda(T) = (\lambda_0(T), \ldots , \lambda_n(T)) \qquad \textrm{for $ T \in X $.}
		\end{equation*}
		We observe that $ \lambda $ is a continuous map by \cite[Theorem B]{MR884486}. Moreover we notice that $ \dmn \Der \bm{\nu}^\phi_K =\dmn \Der \bm{\xi}^\phi_ K $ and that this is a Borel subset of $\mathbf{R}^{n+1}$  by Lemma \ref{lem : Borel measurability differential}. For each $ x \in \dmn \Der \bm{\xi}^\phi_K $, since $ \Der \bm{\xi}^\phi_K(x)(\bm{\nu}^\phi_K(x)) =0 $, we have that 
		\begin{equation*}
			\lambda_0(\Der \bm{\xi}^\phi_K(x)) =0 \qquad \textrm{and} \qquad  \lambda_i(\Der \bm{\xi}^\phi_K(x)) = 1 - \bm{\delta}^\phi_K(x)\rchi^\phi_{K, n+1-i}(x) \geq 0 \quad \textrm{for $ i = 1, \ldots , n $.}
		\end{equation*}
		Since the map $ \Der \bm{\xi}^\phi_K : \dmn \Der \bm{\xi}^\phi_K \rightarrow X $ is a Borel function, we obtain the conclusion.
	\end{proof}

	\begin{Lemma}\label{lem: existence of curvatures}
		Suppose  $ K \subseteq \mathbf{R}^{n+1} $ is closed, $(a,u)\in N^\phi(K) $, $ 0 < r < s < \bm{r}^\phi_K(a,u) $ so that $ a + r u, a+su \in \dmn \Der \bm{\nu}^\phi_K $ and $ \tau_1, \ldots , \tau_n \in \Tan(\mathcal{W}^\phi, u) $. 
		
		Then $ \Der \bm{\nu}^\phi_K(a+ru)\tau_i = \chi^\phi_{K,i}(a+ru)\tau_i $ for every $ i = 1, \ldots , n $ if and only if $ \Der \bm{\nu}^\phi_K(a+su)\tau_i = \chi^\phi_{K,i}(a+su)\tau_i $ for every $ i = 1, \ldots , n $, in which case holds that
		\begin{equation*}
			\frac{\rchi^\phi_{K,i}(a + r u)}{1-r\rchi^\phi_{K,i}(a + r u)} = \frac{\rchi^\phi_{K,i}(a + s u)}{1-s\rchi^\phi_{K,i}(a + s u)}.
		\end{equation*}
	\end{Lemma}
	
	\begin{proof}
		We define $ x = a + ru $, $ y = a + su $ and $ t = \frac{r}{s} $. We notice that $\bm{\xi}^\phi_K $ is differentiable at $ y $ and
		\begin{equation*}
			\Der\bm{\xi}^\phi_K(y)|\Tan(\mathcal{W}^\phi, u) = \textrm{Id}_{\Tan(\mathcal{W}^\phi, u)} - s \Der \bm{\nu}^\phi_K(y)|\Tan(\mathcal{W}^\phi, u).
		\end{equation*}
		Let $ \xi : \mathbf{R}^{n+1} \setminus K \rightarrow K $ and $ \nu : \mathbf{R}^{n+1} \setminus K \rightarrow \mathcal{W}^\phi $ such that $ \xi(z) \in \bm{\xi}^\phi_K(z) $ and $ \nu(z) = \bm{\delta}^\phi_K(z)^{-1}(z - \xi(z)) $ for each $ z \in \mathbf{R}^{n+1} \setminus K $. It follows from the strict convexity of $ \phi $ (see Remark \cite[2.11]{kolasinski2021regularity}) that
		\begin{equation*}
			\nu(\xi(z) + t(z-\xi(z)) ) = \nu(z) \qquad \textrm{for $ z \in \mathbf{R}^{n+1} \setminus K $};
		\end{equation*}
		differentiating this equality in $ y $ we obtain
		\begin{equation*}
			\Der \nu(x) \circ [  \Der \xi(y) + t (\textrm{Id}_{\mathbf{R}^{n+1}} - \Der \xi(y)) ] = \Der \nu(y).
		\end{equation*}
		
		Assume now $ \Der \nu(y)\tau_i = \rchi^\phi_{K,i}(y)\tau_i $ for every $ i = 1, \ldots , n $. Then we get
		\begin{equation*}
			\rchi^\phi_{K,i}(y)\tau_i = [1 - (s-r)\rchi^\phi_{K,i}(y)]\Der \nu(x)\tau_i \qquad \textrm{for every $ i = 1, \ldots , n $.}
		\end{equation*}
		Since $1 - (s-r)\rchi^\phi_{K,i}(y) > 0 $ for every $ i = 1, \ldots , n $ by Lemma \ref{lem: chi curvatures}, we conclude that 
		\begin{equation*}
			\Der \nu(x)\tau_i = \rchi^\phi_{K,i}(x)\tau_i, \qquad  \rchi^\phi_{K,i}(x) = \frac{\rchi^\phi_{K,i}(y)}{1 - (s-r)\rchi^\phi_{K,i}(y)}
		\end{equation*}
		and 
		\begin{equation*}
			\frac{\rchi^\phi_{K,i}(x)}{1-r\rchi^\phi_{K,i}(x)} = \frac{\rchi^\phi_{K,i}(y)}{1-s\rchi^\phi_{K,i}(y)}
		\end{equation*}
		for every $ i = 1 \ldots , n $. 
		
		The last paragraph shows in particular that $ \Der \nu(x)|\Tan(\mathcal{W}^\phi, u) $ and $ \Der \nu(y)|\Tan(\mathcal{W}^\phi, u) $ have the same number $ k $ of distinct eigenvalues. Denoting with $ N_1(x), \ldots , N_k(x) $ and $N_1(y), \ldots , N_k(y) $ the eigenspaces of $ \Der \nu(x)|\Tan(\mathcal{W}^\phi, u) $ and $ \Der \nu(y)|\Tan(\mathcal{W}^\phi, u) $ respectively, we can also derive from the last paragraph the inclusions $ N_i(y) \subseteq N_i(x) $ for every $ i = 1, \ldots , k $. Since 
		\begin{equation*}
			N_1(y) \oplus \cdots \oplus N_k(y) = \Tan(\mathcal{W}^\phi, u) = N_1(x) \oplus \cdots \oplus N_k(x),
		\end{equation*}
		we conclude that $ N_i(y) = N_i(x) $ for every $ i = 1, \ldots , n $ and the proof is completed.
	\end{proof}

	\begin{Definition}
		Suppose $ K \subseteq \mathbf{R}^{n+1} $ is closed.	We define
		\begin{equation*}
			\widetilde{N}^\phi(K) =  \{ (a,u) \in N^\phi(K) : \textrm{$a  + ru  \in \dmn(\Der \bm{\nu}^\phi_K) $ for every $ 0 < r < \bm{r}^\phi_K(a,u) $}\}
		\end{equation*}
		and 
		\begin{equation*}
			\kappa_{K,i}^\phi(a, u) = \frac{\rchi^\phi_{K,i}(a + r u)}{1-r\rchi^\phi_{K,i}(a + r u)} 
		\end{equation*}
		for $(a, u) \in \widetilde{N}^\phi(K) $, $ 0 < r < \bm{r}^\phi_K(a,u) $ and $ i = 1, \ldots , n $.
	\end{Definition}
	
	\begin{Remark}\label{rem: principla curvatures basic rem}
		Lemma \ref{lem: existence of curvatures} ensures that the definition of $ \kappa^\phi_{K,i}(a, u) $ does not depend on the choice of $ r $ and Theorem \ref{theo: distance twice diff} ensures that $ \mathcal{H}^n(N^\phi(K) \setminus \widetilde{N}^\phi(K) ) =0 $. In particular we notice that $ \widetilde{N}^\phi(K) $ is an $ \mathcal{H}^n $-measurable subset of $ N^\phi(K) $. Moreover by Remark \ref{rem: borel meas of chi} we deduce that the functions $ \kappa^\phi_{K,i} $ are $ \mathcal{H}^n \restrict N^\phi(K) $-measurable. 
	\end{Remark}
	\begin{Remark}\label{eq: curvatures and reach}
		Noting \eqref{eq : r and rho},  we conclude from Lemma \ref{lem: chi curvatures} that 
		\begin{equation*}
			-\frac{1}{\bm{r}^\phi_K(a,u)} \leq \kappa^\phi_{K,i}(a,u) \leq + \infty 
		\end{equation*}
		for every $(a,u)\in \widetilde{N}^\phi(K) $ and $ i = 1, \ldots , n $.
	\end{Remark}
	
	\begin{Lemma}\label{lem: tangent of normal bundle}
		Let  $ K \subseteq \mathbf{R}^{n+1} $ be closed. Suppose $ \tau_i : \widetilde{N}^\phi(K)\rightarrow \mathbf{R}^{n+1} $, for $ i = 1, \ldots , n $, are defined so that $ \tau_1(a,u), \ldots, \tau_{n}(a,u) $ form a basis of $\Tan(\mathcal{W}^\phi, u)$ with
		\begin{equation*}
			\Der \bm{\nu}^\phi_K(a+ r u)(\tau_i(a,u))= \rchi^\phi_{K,i}(a + r u)\tau_i(a,u) \qquad \textrm{for $ i = 1, \dots , n $ and $ 0 < r < \bm{r}^\phi_K(a,u) $}
		\end{equation*}
		and $ \zeta_i : \widetilde{N}^\phi(K) \rightarrow \mathbf{R}^{n+1} \times \mathbf{R}^{n+1} $, for $ i = 1, \ldots , n $,  are defined so that 
		\[
		\zeta_i(a,u) =
		\begin{cases}
			(\tau_i(a,u),  \kappa^\phi_{K,i}(a, u)\tau_i(a,u)) & \textrm{if $ \kappa^\phi_{K,i}(a,u) < \infty $}\\
			(0, \tau_i(a,u)) & \textrm{if $\kappa^\phi_{K,i}(a, u) = + \infty$}
		\end{cases}
		\]
		
		For every $ \mathcal{H}^n $-measurable subset $ W \subseteq N^\phi(K) $ with $ \mathcal{H}^n(W) < \infty $ and for $ \mathcal{H}^n $ almost all $(a,u)\in W $ we have that $\Tan^n(\mathcal{H}^n\restrict W, (a,u)) $ is an $ n $-dimensional plane and $\zeta_1(a,u), \ldots , \zeta_n(a,u) $ form a basis of $\Tan^n(\mathcal{H}^n\restrict W, (a,u)) $. Moreover,
		\begin{equation*}
			\ap J^W_n\bm{p}(a,u) = \frac{|\tau_1(a,u) \wedge \ldots \wedge \tau_n(a,u)|}{|\zeta_1(a,u) \wedge \ldots \wedge \zeta_n(a,u)|}\, \bm{1}_{\widetilde{N}^\phi_n(K)}(a,u)
		\end{equation*}
		for $ \mathcal{H}^n $ almost all $(a,u)\in W $.
	\end{Lemma}
	\begin{proof}
		Assume $ W \subseteq N^\phi(K) $ is $ \mathcal{H}^n $-measurable and $ \mathcal{H}^n(W)< \infty $ and $ \lambda > 1 $. For $ r > 0 $ we define 
		\begin{equation*}
			W_r = \{(a,u)\in W:  \bm{r}^\phi_K(a,u) \geq \lambda r  \},
		\end{equation*}
		that is an $ \mathcal{H}^n $-measurable subset of $ W $. We denote with $ W^\ast_r $ the set of all $(a,u)\in W_r $ such that $ \Tan^n(\mathcal{H}^n\restrict W, (a,u)) $ is $ n $-dimensional plane  and $ \Tan(\mathcal{H}^n \restrict W, (a,u))= \Tan^n(\mathcal{H}^n \restrict W_r, (a,u)) $. It follows from \cite[3.2.19]{MR0257325} that $ \mathcal{H}^n(W_r \setminus W^\ast_r) =0 $. Moreover we observe from coarea formula that there exists $ J \subseteq \{t : t >0\} $ with $ \mathcal{H}^1(J) =0 $ such that $ \mathcal{H}^n(S^\phi(K,r) \setminus \Unp^\phi(K)) =0 $ for every $ r \notin J $.
		
		We fix $ r > 0 $, $ r \notin J $, and we define
		\begin{equation*}
			M_r = \{a + ru: (a,u) \in W_r   \}.
		\end{equation*}
		Notice that $ M_r \subseteq S^\phi(K,r) \cap \{x : \bm{\rho}^\phi_K(x)\geq \lambda    \} $ and $ M_r $ is $ \mathcal{H}^n $-measurable. By Theorem \ref{theo: projectionLipschitz} the function $ \bm{\psi}^\phi_K|M_r $ is Lipschitz; moreover we notice that $ \bm{\psi}^\phi_K(M_r) = W_r $ and $ (\bm{\psi}^\phi_K|M_r)^{-1}(a,u) = a + ru $ for $ (a,u)\in W_r $. We denote with $M^\ast_r $ the set of all $ x \in M_r $ such that $ \Tan(S^\phi(K,r), x) $ is an $ n $-dimensional plane and $ \Tan^n(\mathcal{H}^n \restrict M_r,x) = \Tan(S^\phi(K,r), x) $. It follows from Remark \ref{rem: tangent ot level sets} and \cite[3.2.19]{MR0257325} that $ \mathcal{H}^n(M_r \setminus M^\ast_r) =0 $. We conclude that 
		\begin{equation*}
			\mathcal{H}^n(W_r \setminus (W^\ast_r \cap \bm{\psi}^\phi_K(M^\ast_r)))=0.
		\end{equation*}
		Moreover if $(a,u)\in (W^\ast_r \cap \bm{\psi}^\phi_K(M^\ast_r)) \cap \widetilde{N}^\phi(K) $ it follows from \cite[Lemma B.2]{MR4117503} and Remark \ref{rem: tangent level sets and Wulff shapes} that  $ \{\tau_1(a,u), \ldots , \tau_n(a,u) \}$ is a basis of $ \Tan(S^\phi(K,r), a + ru) $,
		\begin{equation*}
			\Der \bm{\psi}^\phi_K(a+ru)[\Tan(S^\phi(K,r), a+ru)] = \Tan^n(\mathcal{H}^n \restrict W, (a,u))
		\end{equation*}
		and
		\[
		\Der \bm{\psi}^\phi_K(a+ru)(\tau_i(a,u)) =
		\begin{cases}
			\frac{1}{1 + r \kappa^\phi_{K,i}(a,u)}\zeta_i(a,u) & \textrm{if $ \kappa^\phi_{K,i}(a,u) < \infty $}\\
			\frac{1}{r}\zeta_i(a,u) & \textrm{if $\kappa^\phi_{K,i}(a, u) = + \infty$.}
		\end{cases}
		\]
		This proves that $\{\zeta_1(a,u), \ldots , \zeta_n(a,u)\}  $ is a basis of $\Tan^n(\mathcal{H}^n \restrict W, (a,u))$ for $ \mathcal{H}^n $ a.e.\ $(a,u)\in W_r $ and for every $ r \notin J $.
		
		Since $ W = \bigcup_{r > 0}W_r $ and $ W_r \subseteq W_r $ for $ 0 < s < r $, there exists a sequence $ r_i \searrow 0 $, $r_i \notin J$, so that $ W = \bigcup_{i=1}^\infty W_{r_i} $. Henceforth, the proof is complete.
	\end{proof}

	\begin{Definition}
		Let $ K \subseteq \mathbf{R}^{n+1} $ be closed.	For every $ d = 0, \ldots , n $ we define 
		\begin{equation*}
			\widetilde{N}^\phi_d(K) = \{ (a,u) \in \widetilde{N}^\phi(K)  : \kappa^\phi_{K, d}(a,u) < \infty ,  \kappa^\phi_{K, d+1}(a,u) = \infty  \}.
		\end{equation*}
		Moreover for every $ j =0, \ldots , n $ we define $ E^\phi_{K,j} : \widetilde{N}^\phi(K) \rightarrow \mathbf{R} $ by
		\[
		E^\phi_{K,j}(a,u) = 
		\begin{cases}
			\sum_{1 \leq l_1 < \ldots < l_j\leq d}\prod_{h=1}^j\kappa^\phi_{K, l_h} & \textrm{if $(a,u) \in \widetilde{N}^\phi_d(K) $ and $ d \geq j $}\\
			0 & \textrm{if $(a,u) \in \widetilde{N}^\phi_d(K) $ and $ d < j $},
		\end{cases}
		\]
		where for $ j =0 $ this means $ E^\phi_{K,0} \equiv 1 $. Finally for every $ r =0, \ldots , n $ we define \emph{$ r$-th $ \phi $-mean curvature of $ K $} as
		\begin{equation*}
			\bm{H}^\phi_{K,r} =\sum_{j=0}^{r}  E^\phi_{K,j}\, \bm{1}_{\widetilde{N}^\phi_{j+n-r}(K)}.
		\end{equation*}
	\end{Definition}

	We can now prove a general Steiner formula for arbitrary closed sets. 
	\begin{Theorem}\label{theo: Steiner closed}
		Suppose	$ K \subseteq \mathbf{R}^{n+1} $ is closed, $ \tau_1, \ldots , \tau_n, \zeta_1, \ldots, \zeta_n $ are $ \mathcal{H}^n \restrict N^\phi(K) $-measurable functions satisfying the hypothesis in Lemma \ref{lem: tangent of normal bundle} and $ J  $ is the $ \mathcal{H}^n \restrict N^\phi(K) $-measurable function defined on $ \mathcal{H}^n $ almost all of $ N^\phi(K) $ by
		\begin{equation*}
			J(a,u) = \frac{|\tau_1(a,u) \wedge \ldots \wedge \tau_n(a,u)|}{|\zeta_1(a,u) \wedge \ldots \wedge \zeta_n(a,u) |} \qquad \textrm{for $ \mathcal{H}^n $ a.e.\ $(a,u)\in N^\phi(K) $.}
		\end{equation*} 
		Then 
		\begin{flalign*}
			& \int_{\{x : 0 < \bm{\delta}^\phi_K(x) \leq \rho\}} (\varphi \circ \bm{\psi}^\phi_K)\, d\mathcal{L}^{n+1} \notag \\
			&  = \sum_{m=0}^n \frac{1}{n-m+1} \int_{N^\phi(K)}\inf\{ \rho, \bm{r}^\phi_K(a,u)  \}^{n+1-m}\phi(\bm{n}^\phi(u))\, J(a,u)\, \bm{H}^\phi_{K,n-m}(a,u)\;\varphi(a,u)\,  d\mathcal{H}^n(a,u)
		\end{flalign*}
		for every bounded Borel function $ \varphi : N^\phi(K) \rightarrow \mathbf{R} $ and for every $ \rho > 0 $.
	\end{Theorem}
	\begin{proof}
		We define $ \Omega =\{(a,u,t): (a,u) \in N^\phi(K), \; 0 < t < \bm{r}^\phi_K(a,u)   \} $ and the bijective map 
		\begin{equation*}
			f : \Omega \rightarrow \mathbf{R}^{n+1} \setminus(K \cup \Cut^\phi(K))
		\end{equation*}
		by $ f(a,u,t) = a + t u $ for $(a,u,t)\in \Omega $.  We choose a countable family $ \{N_i : i \geq 1\} $ of Borel subsets of $ N^\phi(K) $ such that $ \mathcal{H}^n(N_i)< \infty $ and $ N_i \subseteq N_{i+1} $ for every $ i \geq 1 $ and $ N^\phi(K) = \bigcup_{i=1}^\infty N_i $ (recall that $N^\phi(K)$ is a Borel and countably $ n $-rectifiable subset of $ \mathbf{R}^{n+1} \times \mathbf{R}^{n+1} $). We define $\Omega_i = \Omega \cap (N_i \times \mathbf{R}) $ for every $ i \geq 1 $ and we notice that $ \Omega_i $ is a Borel $ (n+1) $-rectifiable subset of $ \mathbf{R}^{n+1} \times \mathbf{R}^{n+1} \times \mathbf{R} $. Let $ \varphi : N^\phi(K) \rightarrow \mathbf{R} $ be a bounded Borel function. Recalling that $ \mathcal{L}^{n+1}(\Cut^\phi(K)) =0 $, we use the dominated convergence theorem and the coarea formula \cite[3.2.22]{MR0257325} to compute
		\begin{flalign*}
			\int_{\{x : 0 < \bm{\delta}^\phi_K(x) \leq \rho\}} (\varphi \circ \bm{\psi}^\phi_K)\, d\mathcal{L}^{n+1} & = \lim_{i \to \infty} \int_{f(\Omega_i) \cap \{x : 0 < \bm{\delta}^\phi_K(x) \leq \rho\}} (\varphi \circ \bm{\psi}^\phi_K)\, d\mathcal{L}^{n+1} \\
			&  = \lim_{i \to \infty} \int_{f[\Omega_i \cap \{(a,u,t) : 0 < t\leq \rho\}]} (\varphi \circ \bm{\psi}^\phi_K)\, d\mathcal{L}^{n+1} \\
			& =  \lim_{i \to \infty} \int_{\Omega_i \cap \{(a,u,t) : 0 < t\leq \rho\}} \varphi(a,u)\,\ap J^{\Omega_i}_{n+1}f(a,u,t)\, d\mathcal{H}^{n+1}(a,u)\\
			& =  \lim_{i \to \infty} \int_{N_i}\int_{0}^{\inf\{ \rho, \bm{r}^\phi_K(a,u)  \}} \varphi(a,u)\,\ap J^{\Omega_i}_{n+1}f(a,u,t)\, dt\, d\mathcal{H}^n(a,u).
		\end{flalign*}
		Noting that 
		\begin{equation*}
			\Tan^{n+1}(\mathcal{H}^{n+1}\restrict \Omega_i, (a,u,t)) = \Tan^n(\mathcal{H}^n\restrict N_i, (a,u))\times \mathbf{R} \qquad \textrm{for $(a,u,t) \in \Omega_i $,}
		\end{equation*}
		we can apply Lemma \ref{lem: tangent of normal bundle} to compute 
		\begin{flalign*}
			& \ap J_{n+1}^{\Omega_i}f(a,u,t)\, \bm{1}_{\widetilde{N}^\phi_d(K)}(a,u) \\
			& \quad = \frac{\big|\Der f(a,u,t)(\zeta_1(a,u), 0)\wedge \ldots \wedge \Der f(a,u,t)(\zeta_n(a,u), 0) \wedge \Der f(a,u,t)(0,1)\big|}{|\zeta_1(a,u) \wedge \ldots \wedge \zeta_n(a,u)|} \, \bm{1}_{\widetilde{N}^\phi_d(K)}(a,u) \\
			& \quad = \frac{|\tau_1(a,u) \wedge \ldots \wedge \tau_n(a,u)\wedge u|}{|\zeta_1(a,u) \wedge \ldots \wedge \zeta_n(a,u)|} t^{n-d}\Bigg(\prod_{j=1}^{d}(1 + t \kappa^\phi_{K,j}(a,u))\Bigg)\, \bm{1}_{\widetilde{N}^\phi_d(K)}(a,u)  \\
			& \quad =(\bm{n}^\phi(u) \bullet u)\frac{|\tau_1(a,u) \wedge \ldots \wedge \tau_n(a,u) \wedge \bm{n}^\phi(u) |}{|\zeta_1(a,u) \wedge \ldots \wedge \zeta_n(a,u)|} t^{n-d} \Bigg(\prod_{j=1}^{d}(1 + t \kappa^\phi_{K,j}(a,u))\Bigg)\, \bm{1}_{\widetilde{N}^\phi_d(K)}(a,u)  \\
			& \quad = \phi(\bm{n}^\phi(u))J(a,u) t^{n-d} \Bigg(\prod_{j=1}^{d}(1 + t \kappa^\phi_{K,j}(a,u))\Bigg)\, \bm{1}_{\widetilde{N}^\phi_d(K)}(a,u)  \\
			& \quad = \phi(\bm{n}^\phi(u)) J(a,u) \sum_{j=0}^d t^{n-d+j}  E^\phi_{K,j}(a,u)\, \bm{1}_{\widetilde{N}^\phi_d(K)}(a,u)
		\end{flalign*}
		for $ \mathcal{H}^{n+1} $ a.e.\ $(a,u,t) \in \Omega_i $ and $ d = 0, \ldots , n $. Consequently 
		\begin{flalign*}
			&\ap J_{n+1}^{\Omega_i}f(a,u,t) \\
			& \qquad = \sum_{d=0}^n \ap J_{n+1}^{\Omega_i}f(a,u,t) \, \bm{1}_{\widetilde{N}^\phi_d(K)} \\
			&\qquad =  \phi(\bm{n}^\phi(u))J(a,u) \sum_{d=0}^n \sum_{j=0}^d t^{n-d+j}   E^\phi_{K,j}(a,u)\, \bm{1}_{\widetilde{N}^\phi_d(K)}(a,u)\\
			& \qquad = \phi(\bm{n}^\phi(u))J(a,u) \sum_{m=0}^n t^{n-m} \bm{H}^\phi_{K, n-m}(a,u)
		\end{flalign*}
		for $ \mathcal{H}^{n+1} $ a.e.\ $(a,u,t)\in \Omega_i $ and  we conclude
		\begin{flalign*}
			& 	\int_{\{x : 0 < \bm{\delta}^\phi_K(x) \leq \rho\}} (\varphi \circ \bm{\psi}^\phi_K)\, d\mathcal{L}^{n+1}\\
			&  =  \sum_{m=0}^n \frac{1}{n-m+1} \int_{N^\phi(K)} \inf\{ \rho, \bm{r}^\phi_K(a,u) \}^{n-m+1}\phi(\bm{n}^\phi(u))\, J(a,u)\, \bm{H}^\phi_{K, n-m}(a,u)\,\varphi(a,u)\, d\mathcal{H}^n(a,u).
		\end{flalign*}
	\end{proof}
	
	\begin{Remark}\label{rem: wel posedness of J}
		By Remark \ref{eq: curvatures and reach} we have that $ 1 + t \kappa^\phi_{K,i}(a,u) > 0 $ for every \ $(a,u) \in \widetilde{N}_n^\phi(K) $ and for every $ 0 < t < \bm{r}^\phi(K) $. Therefore the proof of Theorem \ref{theo: Steiner closed} provides the equality
		\begin{equation*}
			J(a,u) \bm{1}_{\widetilde{N}^\phi_d(K)}(a,u) = \ap J_{n+1}^{\Omega_i}f(a,u,t) \phi(\bm{n}^\phi)^{-1} t^{d-n} \Bigg(\prod_{j=1}^{d}(1 + t \kappa^\phi_{K,j}(a,u))\Bigg)^{-1} \bm{1}_{\widetilde{N}^\phi_d(K)}(a,u)
		\end{equation*}
		for every $ i \geq 1 $ and for $ \mathcal{H}^{n+1} $ a.e.\ $(a,u, t)\in \Omega_i $. In particular, if $ \tau_1', \ldots, \tau'_n, \zeta'_1, \ldots , \zeta'_n $ is another set of $ \mathcal{H}^n \restrict N^\phi(K) $-measurable functions satisfying the hypothesis of Lemma \ref{lem: tangent of normal bundle}, then 
		\begin{equation*}
			\frac{|\tau_1(a,u) \wedge \ldots \wedge \tau_n(a,u)|}{|\zeta_1(a,u) \wedge \ldots \wedge \zeta_n(a,u) |} =  \frac{|\tau'_1(a,u) \wedge \ldots \wedge \tau'_n(a,u)|}{|\zeta'_1(a,u) \wedge \ldots \wedge \zeta'_n(a,u) |}
		\end{equation*}
		for $ \mathcal{H}^n $ a.e.\ $(a,u)\in N^\phi(K) $.
	\end{Remark}

	In view of Remark \ref{rem: wel posedness of J} it is convenient to introduce the following function.
	
	\begin{Definition}
		For every closed set $ K \subseteq \mathbf{R}^{n+1} $, we denote with $ J^\phi_K  $ any $ \mathcal{H}^n \restrict N^\phi(K) $-measurable function defined on $ \mathcal{H}^n $ almost all of $ N^\phi(K) $ such that
		\begin{equation*}
			J^\phi_K(a,u) = \frac{|\tau_1(a,u) \wedge \ldots \wedge \tau_n(a,u)|}{|\zeta_1(a,u) \wedge \ldots \wedge \zeta_n(a,u) |} \qquad \textrm{for $ \mathcal{H}^n $ a.e.\ $(a,u)\in N^\phi(K) $,}
		\end{equation*} 
		where $ \tau_1, \ldots , \tau_n, \zeta_1, \ldots, \zeta_n $ are $ \mathcal{H}^n \restrict N^\phi(K) $-measurable functions satisfying the hypothesis in Lemma \ref{lem: tangent of normal bundle}
	\end{Definition}
	
	\begin{Remark}\label{rem: formula tubular neighb}
		The proof of Theorem \ref{theo: Steiner closed} provides also the formula
		\begin{flalign*}
			&\int_{\{x : 0 < \bm{\delta}^\phi_K(x) \leq \rho\}} (\varphi \circ \bm{\psi}^\phi_K)\, d\mathcal{L}^{n+1}\\
			& \qquad = \sum_{d=0}^{n}\int_{\widetilde{N}^\phi_d(K)}\phi(\bm{n}^\phi(u))J^\phi_K(a,u) \int_{0}^{\inf\{\rho, \bm{r}^\phi_K(a,u)\}}t^{n-d} \Bigg(\prod_{j=1}^{d}(1 + t \kappa^\phi_{K,j}(a,u))\Bigg)\, dt\, d\mathcal{H}^n(a,u)
		\end{flalign*}
		for every $ \rho > 0 $, that will be used in the proof of Theorem \ref{theo: heintze karcher}.
	\end{Remark}
	
	The following Corollary, besides of being of independent interest, plays a key role in the analysis of the equality case in Theorem \ref{theo: heintze karcher}.
	\begin{Corollary}\label{theo: positive reach}
		Suppose $ K \subseteq \mathbf{R}^{n+1} $ be a closed set, $ s > 0 $ and $ \bm{r}^\phi_K(a,u) \geq s $ for $ \mathcal{H}^n $ a.e.\ $(a,u)\in N^\phi(K) $. Then $ \{ x \in \mathbf{R}^{n+1} : \bm{\delta}^\phi_K(x) < s   \} \subseteq \Unp^\phi(K) $. 
	\end{Corollary}
	\begin{proof}
		It follows from Theorem \ref{theo: Steiner closed} that 
		\begin{equation*}
			\int_{\{x : 0 < \bm{\delta}^\phi_K(x) \leq \rho\}} (\varphi \circ \bm{\psi}^\phi_K)\, d\mathcal{L}^{n+1} = \sum_{m=0}^{n+1} \rho^{n+1-m}I_m(\varphi) \qquad \textrm{for $ 0 < \rho < s $,} 
		\end{equation*}
		where $ I_m(\varphi) = \frac{1}{n-m+1} \int_{N^\phi(K)} \phi(\bm{n}^\phi)\, J^\phi_K\, \bm{H}^\phi_{K, n-m}\,\varphi\, d\mathcal{H}^n $ for $ m =0, \ldots , n $. We conclude from \cite[Theorem 5.9]{MR4160798} that $ \{ x \in \mathbf{R}^{n+1} : \bm{\delta}^\phi_K(x) < s   \} \subseteq \Unp^\phi(K) $.
	\end{proof}
	
	We conclude with the following Lemma, that will be useful later.
	
	\begin{Lemma}\label{lem: exterior normal basic properties closed}
		For every closed set $ K \subseteq \mathbf{R}^{n+1} $,
		\begin{equation*}
			\mathcal{H}^0(N^\phi(K,a)) \in \{1,2\} \qquad \textrm{for every $ a \in \bm{p}(\widetilde{N}^\phi_n(K)) $}
		\end{equation*}
		and 
		\begin{equation*}
			\mathcal{H}^n\big(\bm{p}\big[N^\phi(K) \setminus \widetilde{N}_n^\phi(K)\big]\big) =0.
		\end{equation*}
	\end{Lemma}
	\begin{proof}
		Let $ (a,u)\in \widetilde{N}^\phi_n(K) $ and $ 0 < r < \bm{r}^\phi_K(a,u) $. Then $ 1 - r \chi^\phi_{K,i}(a+ru) > 0 $ for every $ i = 1, \ldots , n $ and, since these numbers are the eigenvalues of $ \Der \bm{\xi}^\phi_{K}(a+ ru)|\Tan(S^\phi(K, r), a + ru) $,  we conclude (noting Remark \ref{rem: tangent level sets and Wulff shapes})
		\begin{equation*}
			\Tan(\mathcal{W}^\phi, u) = \Der \bm{\xi}^\phi_{K}(a+ ru)[\Tan(S^\phi(K, r), a + ru)	] \subseteq \Tan(K,a).
		\end{equation*}
		Since $ N(K,a) \subseteq \Nor(K,a) \subseteq \Nor(\mathcal{W}^\phi, u) $ and $ \dim \Nor(\mathcal{W}^\phi, u) = 1 $, it follows that $ 	\mathcal{H}^0(N^\phi(K,a)) \in \{1,2\} $. Moreover combining Lemma \ref{lem: tangent of normal bundle} with coarea formula \cite[3.2.22]{MR0257325} we obtain 
		\begin{equation*}
			0 =	\int_{W \setminus \widetilde{N}^\phi_n(K)} \ap J^W_n \bm{p}(a,u)\, d\mathcal{H}^n(W) = \int_{\mathbf{p}[W \setminus \widetilde{N}^\phi_n(K)]} \mathcal{H}^0(N^\phi(K,x))\, d\mathcal{H}^n(x)
		\end{equation*}
		\begin{equation*}
			\mathcal{H}^n\big(\bm{p}[W \setminus \widetilde{N}^\phi_n(K)]\big) =0
		\end{equation*}
		for every $ \mathcal{H}^n $-measurable subset $ W \subseteq N^\phi(K) $ with $ \mathcal{H}^n(W) < \infty $. 
	\end{proof}

	\begin{Remark}\label{rem normal cone convex}
		In particular it follows from the previous Remark that if $ K \subseteq \mathbf{R}^{n+1} $ is a convex body, then $\mathcal{H}^0(N^\phi(K,a)) = 1 $ for each $ a \in \bm{p}(\widetilde{N}_n^\phi(K)) $.
	\end{Remark}

	\section{Anisotropic curvature measures for convex sets}\label{Section: convex}

	If $ K \subseteq \mathbf{R}^{n+1} $ is a closed convex set, then $ \bm{r}^\phi_K \equiv + \infty $  and $ \kappa^\phi_{K,i} \geq 0 $ for every $ i = 1, \ldots , n $ by \eqref{eq: curvatures and reach}. Moreover $ N^\phi(K) $ is an $ n $-dimensional compact Lipschitz manifold and $ \Cut^\phi(K) = \varnothing $. We introduce now the anisotropic curvature measures of a convex set.
	\begin{Definition}
		Let $ K \subseteq \mathbf{R}^{n+1} $ be a closed convex set and $ m =0, \ldots , n $. The \emph{$m$-th support measure of $ K $ with respect to $ \phi $} is the measure $\widetilde{\mathcal{C}}^\phi_{m}(K,\cdot)$ over $ \mathbf{R}^{n+1} \times \mathbf{R}^{n+1} $ defined by (see \cite[2.4]{MR0257325} for the definition of the upper integral $ \int^\ast $)
		\begin{equation*}
			\widetilde{\mathcal{C}}^\phi_{m}(K,B) = \frac{1}{n-m+1}\int^\ast_{B \cap N^\phi(K)} \phi(\bm{n}^\phi(u))\,J^\phi_K(a,u)\, \bm{H}^\phi_{K,n-m}(a,u)\; d\mathcal{H}^n(a,u) 
		\end{equation*}
		for every $ B \subseteq \mathbf{R}^{n+1} \times \mathbf{R}^{n+1}$ and the \emph{$ m $-th curvature measure of $ K $ with respect to $ \phi $} is the measure over $ \mathbf{R}^{n+1} $ given by (see \cite[2.1.2]{MR0257325})
		\begin{equation*}
			\mathcal{C}^\phi_m(K,\cdot) = \mathbf{p}_{\#}	\widetilde{\mathcal{C}}^\phi_{m}(K,\cdot).
		\end{equation*}
		We also write $ \mathcal{C}^\phi_m(K) = \mathcal{C}^\phi_m(K, \mathbf{R}^{n+1}) $ (this is the anisotropic $ m $-th mixed volume of $ K $).
	\end{Definition}
	\begin{Remark}
		The measure $\widetilde{\mathcal{C}}^\phi_{m}(K,\cdot)$ is a Radon measure over $ \mathbf{R}^{n+1}\times \mathbf{R}^{n+1} $ (see \cite[2.2.5]{MR0257325}). Therefore it follows from \cite[2.2.17]{MR0257325} that $ C^\phi_m(K, \cdot) $ is a Radon measure over $ \mathbf{R}^{n+1} $.
	\end{Remark}
	We can now state the \emph{local anisotropic Steiner formula for convex sets}, which readily follows from Theorem \ref{theo: Steiner closed}. 
	\begin{Corollary}\label{theo: Steiner convex}
		Let $ K \subseteq \mathbf{R}^{n+1} $ be a closed convex set. Then
		\begin{equation*}
			\mathcal{L}^{n+1}(\{x \in \mathbf{R}^{n+1}: 0 < \bm{\delta}^\phi_K(x)\leq \rho, \; \bm{\xi}^\phi_K(x) \in B   \}) = \sum_{m=0}^n \rho^{n+1-m} \mathcal{C}^\phi_m(K,B)
		\end{equation*}
		for every Borel subset $ B \subseteq \partial K $ and for every $ \rho > 0 $.
	\end{Corollary}

	\begin{Lemma}\label{lem: intergal of H0}
		Suppose $ K \subseteq \mathbf{R}^{n+1} $ is a convex body and $ \eta $ is its (measure theoretic) exterior unit normal. Then 
		\begin{equation*}
			\mathcal{C}^\phi_n(K, B) = \int_{B}\phi(\eta(x))\; d\mathcal{H}^n(x)
		\end{equation*}
		for every Borel set $ B \subseteq \partial K $.
	\end{Lemma}
	
	\begin{proof}
		Noting that $ \bm{H}^\phi_{K,0} = \bm{1}_{\widetilde{N}^\phi_n(K)} $, Lemma \ref{lem: tangent of normal bundle} implies that
		\begin{equation*}
			\ap J^{N^\phi(K)}_n \mathbf{p}(a,u)	 = J^\phi_K(a,u) \cdot \bm{H}^\phi_{K,0}(a,u) \qquad \textrm{for $ \mathcal{H}^n $ a.e.\ $(a,u)\in N^\phi(K) $.}
		\end{equation*} 
		Therefore, noting that $ N^\phi(K,x) = \{\nabla \phi(\eta(x))\} $ for $ \mathcal{H}^n $ a.e.\ $ x \in \partial K $ by \eqref{eq: phi normal vs euclidean normal}, we apply Coarea formula \cite[3.2.22]{MR0257325} to conclude
		\begin{flalign*}
			\mathcal{C}^\phi_n(K,B) & = \int_{N^\phi(K)|B}\ap J^{N^\phi(K)}_n \mathbf{p}(a,u)\,\phi(\bm{n}^\phi(u)) \, d\mathcal{H}^n(a,u) \\
			& = \int_{B \cap \partial K} \int_{N^\phi(K,x)} \phi(\bm{n}^\phi(u))\,d\mathcal{H}^0(u)\,  d\mathcal{H}^n(x)\\
			& = \int_{B \cap \partial K}\phi(\eta(x))\; d\mathcal{H}^n(x)
		\end{flalign*}
	\end{proof}
	
	We now prove the anisotropic Minkowski formulae for arbitrary convex bodies.
	\begin{Theorem}\label{theo: minkowski formula}
		If $ K \subseteq \mathbf{R}^{n+1} $ be a convex body and $ r = 1, \ldots , n $ then 
		\begin{flalign*}
			& (n-r+1)\int_{N^\phi(K)} \phi(\bm{n}^\phi(u))\, J^\phi_K(a,u)\, \bm{H}^\phi_{K, r-1}(a,u)\, d\mathcal{H}^n(a,u) \\
			& \qquad  = r\int_{N^\phi(K)} [a \bullet \bm{n}^\phi(u)]J^\phi_K(a,u)\, \bm{H}^\phi_{K,r}(a,u)\, d\mathcal{H}^n(a,u).
		\end{flalign*} 
	\end{Theorem}
	
	\begin{proof}
		We set $ \eta(x) = \frac{\nabla \bm{\delta}^\phi_K(x)}{|\nabla \bm{\delta}^\phi_K(x)|} $ for $ x \in \mathbf{R}^{n+1} \setminus K $ and 
		\begin{equation*}
			B_\rho(K) = \{ x \in \mathbf{R}^{n+1} : 0 \leq \bm{\delta}^\phi_K(x) \leq \rho  \} \qquad \textrm{for $  0 < \rho < \infty $.}
		\end{equation*}
		We notice that $ B_\rho(K) $ is a convex body with $ \mathcal{C}^{1,1} $-boundary $ \partial B_\rho(K) = S^\phi(K, \rho) $ and $ \eta|\partial B_\rho(K) $ is its exterior unit normal. Moreover for each $ \rho > 0 $ the map $ f_\rho : N^\phi(K) \rightarrow S^\phi(K,\rho) $, defined as
		\begin{equation*}
			f_\rho(a,u) = a + \rho u \qquad \textrm{for $(a,u)\in N^\phi(K) $,}
		\end{equation*}
		is a bi-lipschitz homeomorphism by Theorem \ref{theo: projectionLipschitz}. We observe (see proof of Theorem \ref{theo: Steiner closed}) that 
		\begin{equation*}
			J_n^{N^\phi(K)}f_\rho(a,u) =J^\phi_K(a,u) \sum_{m=0}^n \rho^{n-m}\, \bm{H}^\phi_{K,n-m}(a,u)
		\end{equation*}
		for $ \mathcal{H}^n $ a.e.\ $(a,u)\in N^\phi(K) $. We set
		\begin{equation*}
			I_m(K) = \int_{N^\phi(K)}a \bullet \bm{n}^\phi(u)\,J^\phi_K(a,u)\, \bm{H}^\phi_{K,n-m}(a,u) \, d\mathcal{H}^n(a,u) \qquad \textrm{for $ m =0, \ldots , n $}
		\end{equation*}
		and we compute
		\begin{flalign*}
			(n+1)\mathcal{L}^{n+1}(B_\rho(K)) & = \int_{S^\phi(K,\rho)} x \bullet \eta(x)\, d\mathcal{H}^n(x)\\
			& = \int_{N^\phi(K)}[(a+ \rho u)\bullet \eta(a+ \rho u)] \, J_n^{N^\phi(K)}f_\rho(a,u)\, d\mathcal{H}^n(a,u)\\
			& = \sum_{m=0}^n \rho^{n-m}I_m(K) +\sum_{m=0}^n (n-m+1)\rho^{n-m +1}\mathcal{C}^\phi_m(K).
		\end{flalign*}
		Employing the Steiner formula \ref{theo: Steiner convex} we get 
		\begin{equation*}
			(n+1)\mathcal{L}^{n+1}(B_\rho(K)) = (n+1)\mathcal{L}^{n+1}(K) + (n+1)\sum_{m=0}^n \rho^{n-m+1}\mathcal{C}^\phi_m(K)
		\end{equation*}
		and we infer
		\begin{equation*}
			\sum_{m=0}^{n-1}[I_m(K) - (m+1)\mathcal{C}^\phi_{m+1}(K)]\rho^{n-m} + I_n(K) - (n+1)\mathcal{L}^{n+1}(K) = 0
		\end{equation*}
		for every $ \rho > 0 $. It follows that $ I_m(K) = (m+1)\mathcal{C}^\phi_{m+1}(K) $ for $ m =0, \ldots , n-1 $ and $ I_n(K) = (n+1)\mathcal{L}^{n+1}(K) $.
	\end{proof}

	\begin{Lemma}\label{lem: constant mean curvature properties}
		Suppose $ K \subseteq \mathbf{R}^{n+1} $ is a convex body, $ r= 1, \ldots , n $, $ \lambda > 0 $ and $ \mathcal{C}^\phi_{n-r}(K, \cdot) = \lambda \mathcal{C}^\phi_n(K, \cdot) $. 
		
		Then $\bm{H}^\phi_{K,r}(a,u) = 0$ for $ \mathcal{H}^n $ a.e.\ $(a,u)\in N^\phi(K) \setminus \widetilde{N}^\phi_n(K) $ and 
		\begin{equation*}
			\bm{H}^\phi_{K,r}(a,u) = (r+1) \lambda \geq \bigg(\frac{\mathcal{C}^\phi_n(K)}{(n+1)\mathcal{L}^{n+1}(K)}\bigg)^r {n \choose r} \qquad \textrm{for $ \mathcal{H}^n $ a.e.\ $(a,u)\in \widetilde{N}^\phi_n(K) $.}
		\end{equation*} 
	\end{Lemma}
	
	\begin{proof}
		We notice from Remark \ref{rem normal cone convex} that $ \mathcal{H}^0(N^\phi(K,a))=1 $ for each $ a \in \bm{p}(\widetilde{N}_n^\phi(K)) $ and
		\begin{equation*}
			\bm{p}[N^\phi(K) \setminus \widetilde{N}^\phi_n(K)] \cap \bm{p}[\widetilde{N}^\phi_n(K)] = \varnothing.
		\end{equation*}
		From the equality $ \mathcal{C}^\phi_{n-r}(K, \cdot) = \lambda 
		\mathcal{C}^\phi_n(K, \cdot) $ we get 
		\begin{flalign*}
			&\int^\ast_{\widetilde{N}_n^\phi(K)|B} \phi(\bm{n}^\phi(u))\,J^\phi_K(a,u)\, \Big(\frac{1}{r+1}\bm{H}^\phi_{K,r}(a,u) - \lambda\Big)\; d\mathcal{H}^n(a,u) \\
			& \qquad + \frac{1}{r+1} \int^\ast_{(N^\phi(K) \setminus \widetilde{N}_n^\phi(K))|B} \phi(\bm{n}^\phi(u))\,J^\phi_K(a,u)\, \bm{H}^\phi_{K, r}(a,u)\; d\mathcal{H}^n(a,u) =0
		\end{flalign*}
		for every subset $ B \subseteq \partial K $. Choosing $ B = \bm{p}[N^\phi(K) \setminus \widetilde{N}^\phi_n(K)] $ in the last equation, we infer that $\mathbf{H}^\phi_{K,r}(a,u) = 0$ for $ \mathcal{H}^n $ a.e.\ $(a,u)\in N^\phi(K) \setminus \widetilde{N}^\phi_n(K) $. Moreover, 
		\begin{equation*}
			\int^\ast_{\widetilde{N}_n^\phi(K)|B} \phi(\bm{n}^\phi(u))\,J^\phi_K(a,u)\, \Big(\frac{1}{r+1}\bm{H}^\phi_{K,r}(a,u) - \lambda\Big)\; d\mathcal{H}^n(a,u) =0
		\end{equation*}
		for every $ B \subseteq \bm{p}[\widetilde{N}^\phi_n(K)] $ and we infer that  $\frac{1}{r+1}\mathbf{H}^\phi_{K,r}(a,u) = \lambda $ for $ \mathcal{H}^n $ a.e.\ $(a,u)\in \widetilde{N}^\phi_n(K) $.
		
		Let $ \eta $ be the (measure theoretic) exterior unit normal of $ K $. Then, using Lemma \ref{lem: tangent of normal bundle} and coarea formula, we observe
		\begin{flalign*}
			&\int_{N^\phi(K)} [a \bullet \bm{n}^\phi(u)]J^\phi_K(a,u)\, \bm{H}^\phi_{K,r}(a,u)\, d\mathcal{H}^n(a,u) \\
			& \qquad = \int_{\widetilde{N}_n^\phi(K)} [a \bullet \bm{n}^\phi(u)]J^\phi_K(a,u)\, \bm{H}^\phi_{K,r}(a,u)\, d\mathcal{H}^n(a,u) \\
			& \qquad  = (r+1)\lambda \int_{\widetilde{N}_n^\phi(K)} [a \bullet \bm{n}^\phi(u)]J^\phi_K(a,u)\, d\mathcal{H}^n(a,u)\\
			& \qquad  = (r+1)\lambda \int_{N^\phi(K)} [a \bullet \bm{n}^\phi(u)]\ap J^{N^\phi(K)}_n \mathbf{p}(a,u) d\mathcal{H}^n(a,u) \\
			& \qquad  = (r+1)\lambda \int_{\partial K}\int_{N^\phi(K,a)}[a \bullet \bm{n}^\phi(u)]\, d\mathcal{H}^0 u \, d\mathcal{H}^na\\
			& \qquad  = (r+1)\lambda \int_{\partial K} [a \bullet \eta(a)]\, d\mathcal{H}^n a\\
			& \qquad  = (r+1) \lambda (n+1)\mathcal{L}^{n+1}(K).
		\end{flalign*}
		Moreover, employing the Newton-McLaurin inequality \cite[Theorem 1.1]{MR1786404}, we obtain
		\begin{flalign*}
			&\int_{N^\phi(K)} \phi(\bm{n}^\phi(u))\, J^\phi_K(a,u)\, \bm{H}^\phi_{K, r-1}(a,u)\, d\mathcal{H}^n(a,u) \\
			& \qquad \geq \int_{\widetilde{N}_n^\phi(K)} \phi(\bm{n}^\phi(u))\, J^\phi_K(a,u)\, \bm{H}^\phi_{K, r-1}(a,u)\, d\mathcal{H}^n(a,u) \\
			& \qquad \geq [(r+1)\lambda]^{\frac{r-1}{r}}{n \choose r}^{\frac{1-r}{r}}{n \choose r-1}\int_{\widetilde{N}_n^\phi(K)} \phi(\bm{n}^\phi(u))\, J^\phi_K(a,u)\, d\mathcal{H}^n(a,u) \\
			& \qquad =[(r+1)\lambda]^{\frac{r-1}{r}} {n \choose r}^{\frac{1-r}{r}}{n \choose r-1}\mathcal{C}^\phi_n(K).
		\end{flalign*}
		We now use the Minkowski Formula in \ref{theo: minkowski formula} and, noting that $ (n-r+1){n \choose r}^{-1}{n \choose r-1}\frac{1}{r} = 1 $, we conclude
		\begin{equation*}
			[\lambda (r+1)]^{\frac{1}{r}} \geq \frac{\mathcal{C}^\phi_n(K)}{(n+1)\mathcal{L}^{n+1}(K)} {n \choose r}^{\frac{1}{r}}.
		\end{equation*}
	\end{proof}

	\section{An optimal geometric inequality}\label{Section: heintze karcher}
	
	In this section we assume  that $ \phi $ is a uniformly convex $ \mathcal{C}^2 $ norm.
	
	\begin{Lemma}\label{lem: convex body and its complementary}
		Suppose $ C \subseteq \mathbf{R}^{n+1} $ is a convex body, $ K $ is the closure of $ \mathbf{R}^{n+1} \setminus C $ and
		\begin{equation*}
			\iota : \mathbf{R}^{n+1} \times \mathbf{R}^{n+1} \rightarrow \mathbf{R}^{n+1} \times \mathbf{R}^{n+1}
		\end{equation*}
		is the linear map defined as $ \iota(a,u) = (a, -u) $ for every $(a,u)\in \mathbf{R}^{n+1} \times \mathbf{R}^{n+1}$. 
		
		Then the following statements hold.
		\begin{enumerate}[(a)]
			\item\label{lem: convex body and its complementary 1} $ \mathcal{H}^0(N^\phi(K,a)) = 1 $ and $ N^\phi(K,a) = - N^\phi(C,a) $ for every $ a \in \bm{p}(N^\phi(K)) $.
			\item\label{lem: convex body and its complementary 2} $ \mathcal{H}^n\big(N^\phi(K) \setminus \widetilde{N}^\phi_n(K)) = 0 $.
			\item\label{lem: convex body and its complementary 3} $\kappa^\phi_{K,i}(a, u) = - \kappa^\phi_{C,n+1-i}(a, -u) $ for $ \mathcal{H}^n $ a.e.\ $ (a,u) \in \widetilde{N}_n^\phi(K) $ and $ i = 1,\ldots , n $.
			\item\label{lem: convex body and its complementary 4} $J^\phi_K(a,u) = \ap J_n^{N^\phi(K)}\iota(a,u)\, J^\phi_C(a, -u)$ for $ \mathcal{H}^n $ a.e.\ $(a,u)\in N^\phi(K) $.
		\end{enumerate}
	\end{Lemma}
	
	\begin{proof}
		The statement in \ref{lem: convex body and its complementary 1} is contained in \cite[Lemma 5.1]{santilli2020uniqueness} and the statement in \ref{lem: convex body and its complementary 2} follows from \cite[Lemma 5.1]{santilli2020uniqueness} noting that $\mathcal{H}^n\big(\bm{p}\big[N^\phi(K) \setminus \widetilde{N}_n^\phi(K)\big]\big) =0 $ by Lemma \ref{lem: exterior normal basic properties closed}.
		
		We prove \ref{lem: convex body and its complementary 3}. Firstly, we prove that
		\begin{equation}\label{lem: convex body and its complementary: 5}
			\bm{\nu}^\phi_K(y) = \{ - \bm{\nu}^\phi_C((1+\lambda)a-\lambda y): a \in \bm{\xi}^\phi_K(y)  \} \qquad \textrm{for every $ y \in \mathbf{R}^{n+1} \setminus K $ and $ \lambda > 0 $.}
		\end{equation}
		The equality in \eqref{lem: convex body and its complementary: 5} follows noting that if $ y \in \mathbf{R}^{n+1} \setminus K $, $ \lambda > 0 $ and $ a \in \bm{\xi}^\phi_K(y) $, then
		\begin{equation*}
			a + \lambda\bm{\delta}^\phi_K(y)\frac{a-y}{\bm{\delta}^\phi_K(y)} = (1+\lambda)a-\lambda y, 
		\end{equation*} 
		\begin{equation*}
			N^\phi(K,a)= \bigg\{ \frac{y-a}{\bm{\delta}^\phi_K(y)}\bigg\}, \qquad N^\phi(C,a)= \bigg\{ \frac{a-y}{\bm{\delta}^\phi_K(y)}\bigg\},
		\end{equation*}
		\begin{equation*}
			\bm{\nu}^\phi_C((1+\lambda)a-\lambda y) = \frac{a-y}{\bm{\delta}^\phi_K(y)}, \qquad -\bm{\nu}^\phi_C((1+\lambda)a-\lambda y) \in \bm{\nu}^\phi_K(y).
		\end{equation*}
		Define  $ S = \bm{p}\big( \iota(N^\phi(K)) \setminus \widetilde{N}^\phi(C)\big) $ and notice that $ \mathcal{H}^n(S) =0 $ by Remark \ref{rem: principla curvatures basic rem}. It follows from \cite[Lemma 5.1]{santilli2020uniqueness} that $ \mathcal{H}^n(N^\phi(K)|S) =0 $. Fix now $(a,u)\in \widetilde{N}^\phi_n(K) $ with $ a \notin S $, $ 0 < r < \bm{r}^\phi_K(a,u) $ and, noting that $ 1 - r \chi^\phi_{K,i}(a+ru) > 0 $ for every $ i = 1, \ldots, n $,  we select  $ \lambda > 0 $ so that 
		\begin{equation*}
			\chi^\phi_{K,i}(a+ ru) < \frac{1}{(1+\lambda)r} \qquad \textrm{for $ i = 1, \ldots , n $.}
		\end{equation*}
		Since $ a \notin S $, then $(a, -u)\in \widetilde{N}^\phi(C) $ and $ \bm{\nu}^\phi_C $ is differentiable at $ a - tu $ for every $ t > 0 $. Differentiating at $ a + ru $ the equality in \eqref{lem: convex body and its complementary: 5}, we compute
		\begin{equation*}
			\Der \bm{\nu}^\phi_K(a+ ru) = - \Der \bm{\nu}^\phi_C(a- \lambda ru) \circ ((1+\lambda)\Der \bm{\xi}^\phi_K(a+ru) - \textrm{I}_{\mathbf{R}^{n+1}}).
		\end{equation*}
		If $ \tau_1, \ldots , \tau_n $ form a basis of $ \Tan(\mathcal{W}^\phi, u)  $ such that $\Der \bm{\nu}^\phi_K(a+ ru)(\tau_i) = \chi^\phi_{K,i}(a+ru)\tau_i $ for every $ i = 1, \ldots , n $, then we infer
		\begin{equation*}
			\Der \bm{\nu}^\phi_C(a- \lambda ru)(\tau_i) = \frac{\chi^\phi_{K,i}(a+ru)}{(1+\lambda)r\chi^\phi_{K,i}(a+ru) -1} \tau_i \qquad \textrm{for $ i = 1, \ldots , n $}
		\end{equation*}
		and, noting that $ \Tan(\mathcal{W}^\phi_1, -u) = \Tan(\mathcal{W}^\phi_1, u) $ and $(1+\lambda)r\chi^\phi_{K,i}(a+ru) -1 < 0 $ for every $  i = 1, \ldots  n $, we conclude that
		\begin{equation*}
			\chi^\phi_{C, n+1-i}(a- \lambda ru) = \frac{\chi^\phi_{K,i}(a+ru)}{(1+\lambda)r\chi^\phi_{K,i}(a+ru) -1}
		\end{equation*}
		for every $ i = 1, \ldots , n $. Therefore, 
		\begin{equation*}
			\kappa^\phi_{C, n+1-i}(a,-u) = \frac{\chi^\phi_{C, n+1-i}(a- \lambda ru)}{1 - \lambda r \chi^\phi_{C, n+1-i}(a- \lambda ru)} = \frac{\chi^\phi_{K,i}(a+ ru)}{r \chi^\phi_{K,i}(a+ ru) -1} = - \kappa^\phi_{K,i}(a,u)
		\end{equation*}
		for every $ i = 1, \ldots, n $.
		
		Finally we prove \ref{lem: convex body and its complementary 4}. Let $ \tau_1, \ldots , \tau_n, \zeta_1, \ldots , \zeta_n $ be $ \mathcal{H}^n \restrict N^\phi(K) $-measurable functions satisfying the hypothesis of Lemma \ref{lem: tangent of normal bundle}. The argument of the previous paragraph in combination with Lemma \ref{lem: existence of curvatures} shows that 
		\begin{equation*}
			\Der \bm{\nu}^\phi_C(a - tu)(\tau_i(a,u)) = \rchi^\phi_{C, n+1-i}(a-tu) \tau_i(a,u)
		\end{equation*}
		for $ \mathcal{H}^n $ a.e.\ $(a,u)\in N^\phi(K) $ and for every $ t > 0 $. Since $ \kappa^\phi_{C,i}(a, -u) < \infty $ for $ \mathcal{H}^n $ a.e.\ $ (a,u)\in N^\phi(K) $ by \ref{lem: convex body and its complementary 2} and \ref{lem: convex body and its complementary 3}, we infer that 
		\begin{equation*}
			J^\phi_C(a,-u) = \frac{|\tau_1(a,u)\wedge \ldots \wedge \tau_n(a,u)|}{|\iota(\zeta_1(a,u))\wedge \ldots \wedge \iota(\zeta_n(a,u))|}
		\end{equation*}
		for $\mathcal{H}^n$ a.e.\ $(a,u)\in N^\phi(K) $. Since
		\begin{equation*}
			\ap J_n^{N^\phi(K)}\iota(a,u) = \frac{|\iota(\zeta_1(a,u))\wedge \ldots \wedge \iota(\zeta_n(a,u))|}{|\zeta_1(a,u)\wedge \ldots \wedge \zeta_n(a,u)|},
		\end{equation*}
		the equation in \ref{lem: convex body and its complementary 4} follows.
	\end{proof}

	\begin{Theorem}\label{theo: heintze karcher}
		Suppose $ C \subseteq \mathbf{R}^{n+1} $ is a convex body. Then 
		\begin{equation*}
			(n+1)\mathcal{L}^{n+1}(C) \leq n	\int_{\widetilde{N}_n^\phi(C)} J^\phi_C(a,u)\frac{\phi(\bm{n}^\phi(u))}{\bm{H}^\phi_{C,1}(a,u)}\, d\mathcal{H}^n(a,u).
		\end{equation*}
		Moreover, if the equality holds and there exists $ q < \infty $ so that $\bm{H}^\phi_{C,1}(a,u) \leq q $ for $ \mathcal{H}^n $ a.e.\ $(a,u)\in \widetilde{N}_n^\phi(C) $, then $ \partial C = a + s \mathcal{W}^\phi $ for some $ a \in \mathbf{R}^{n+1} $ and $ s > 0 $.
	\end{Theorem}
	
	\begin{proof}
		We assume that $ \bm{H}^\phi_{C,1}(a,u) > 0 $ for $ \mathcal{H}^n $ a.e.\ $(a,u)\in \widetilde{N}^\phi_n(C) $, otherwise the inequality is trivially true. Let $ K $ be the closure of $ \mathbf{R}^{n+1} \setminus C $ and notice by Lemma \ref{lem: convex body and its complementary} that 
		\begin{flalign*}
			\bm{H}^\phi_{K,1}(a,u) & = \sum_{i=1}^n \kappa^\phi_{K,i}(a,u) \\
			& = - \sum_{i=1}^n \kappa^\phi_{C,n+1-i}(a,-u) = - \bm{H}^\phi_{C,1}(a,-u) <0 
		\end{flalign*}
		for $ \mathcal{H}^n $ a.e.\ $(a,u)\in N^\phi(K) $. We use the inequality in Remark \ref{eq: curvatures and reach} to infer that
		\begin{equation*}
			\bm{r}^\phi_K(a,u) \leq -\frac{n}{\bm{H}^\phi_{K,1}(a,u)} \qquad \textrm{for $ \mathcal{H}^n $ a.e.\ $(a,u)\in N^\phi(K) $}
		\end{equation*}
		and $ 1 + t \kappa^\phi_{K,i}(a,u) >0 $ for $ \mathcal{H}^n $ a.e.\ $ (a,u)\in  N^\phi(K) $ and for every $ 0 < t < \bm{r}^\phi_K(a,u) $. 
		Noting that $ \bm{n}^\phi(u) = - \bm{n}^\phi(-u) $ for every $ u \in \mathcal{W}^\phi $, we can use Lemma \ref{lem: convex body and its complementary}, the formula in Remark \ref{rem: formula tubular neighb} and the classical arithmetic-geometric mean inequality to estimate
		\begin{flalign*}
			\mathcal{L}^{n+1}(C) & = \int_{N^\phi(K)}\phi(\bm{n}^\phi(u))J^\phi_K(a,u)\,  \int_{0}^{\bm{r}^\phi_K(a,u)} \prod_{j=1}^{n}(1 + t \kappa^\phi_{K,j}(a,u))\, dt\, d\mathcal{H}^n(a,u)\\
			& \leq \int_{N^\phi(K)}\phi(\bm{n}^\phi(u))J^\phi_K(a,u)\,  \int_{0}^{\bm{r}^\phi_K(a,u)} \Big( 1 + \frac{t}{n}\bm{H}^\phi_{K,1}(a,u)\Big)^n\, dt\, d\mathcal{H}^n(a,u) \\
			& \leq \int_{N^\phi(K)}\phi(\bm{n}^\phi(u))J^\phi_K(a,u) \int_{0}^{-\frac{n}{\bm{H}^\phi_{K,1}(a,u)}} \Big( 1 + \frac{t}{n}\bm{H}^\phi_{K,1}(a,u)\Big)^n\, dt\, d\mathcal{H}^n(a,u)\\
			& = \int_{N^\phi(K)}\ap J_n^{N^\phi(K)}\iota(a,u)\,J^\phi_C(a,-u)\, \phi(\bm{n}^\phi(u))  \, \int_{0}^{-\frac{n}{\bm{H}^\phi_{K,1}(a,u)}} \Big( 1 + \frac{t}{n}\bm{H}^\phi_{K,1}(a,u)\Big)^n\, dt\, d\mathcal{H}^n(a,u)\\
			& \leq \int_{\widetilde{N}_n^\phi(C)}\phi(\bm{n}^\phi(u)) J^\phi_C(a,u) \int_{0}^{\frac{n}{\bm{H}^\phi_{C,1}(a,u)}} \Big( 1 - \frac{t}{n}\bm{H}^\phi_{C,1}(a,u)\Big)^n\, dt\, d\mathcal{H}^n(a,u)\\
			& = \frac{n}{n+1}\int_{\widetilde{N}_n^\phi(C)}J^\phi_C(a,u)\, \frac{\phi(\bm{n}^\phi(u)) }{\bm{H}^\phi_{C,1}(a,u)}\, d\mathcal{H}^n(a,u).
		\end{flalign*}
		
		We discuss now the equality case. We assume that $\bm{H}^\phi_{C,1}(a,u) \leq q $ for $ \mathcal{H}^n $ a.e.\ $(a,u)\in \widetilde{N}_n^\phi(C) $. We notice that the inequalities in the last estimate become equalities. In particular we deduce the inequality
		\begin{equation*}
			\bm{r}^\phi_K(a,u) = -\frac{n}{\bm{H}^\phi_{K,1}(a,u)} \geq \frac{n}{q} \qquad \textrm{for $ \mathcal{H}^n $ a.e.\ $(a,u)\in N^\phi(K) $}
		\end{equation*} 
		and the umbilicality condition
		\begin{equation*}
			\kappa^\phi_{K,1}(a,u) = \ldots = \kappa^\phi_{K,n}(a,u) \qquad \textrm{for $ \mathcal{H}^n $ a.e.\ $(a,u)\in N^\phi(K) $}.
		\end{equation*}
		Consequently we infer from Corollary \ref{theo: positive reach} that $ \{ x \in \mathbf{R}^{n+1} : \bm{\delta}^\phi_K(x) < \frac{n}{q}  \} \subseteq \Unp^\phi(K) $. It follows that $ S^\phi(K,r)  $ is a closed $ \mathcal{C}^{1,1} $-hypersurface for every $ 0 < r < \frac{n}{q} $ by \cite[Corollary 5.8]{MR4160798}. Moreover the aforementioned umbilicality condition gives
		\begin{equation*}
			\rchi^\phi_{K,1}(x) = \ldots = \rchi^\phi_{K,n}(x)
		\end{equation*}
		for $ \mathcal{H}^n $ a.e.\ $x \in S^\phi(K,r) $ and for every $ 0 < r < \frac{n}{q} $. We deduce from \cite[Lemma 3.2]{MR4160798} that for every $ 0 < r < \frac{n}{q} $ there exists $ c_r \in \mathbf{R}^{n+1} $  and $ \lambda_r > 0 $ so that $ S^\phi(K,r) = c_r + \lambda_r \mathcal{W}^\phi $. (Notice that the last line of \cite[Lemma 3.2]{MR4160798}  contains a typo: one should replace the equality $ M = \partial \mathbf{B}^F(a, |\lambda|^{-1}) $ with $ M = \partial \mathbf{B}^{F^\ast}(a, |\lambda|^{-1}) $, which is what the proof given there proves.) Now fix $ 0 < r < \frac{n}{q} $ and notice that
		\begin{equation*}
			\frac{\nabla \bm{\delta}^\phi_K(c_r + \lambda_r z)}{|\nabla \bm{\delta}^\phi_K(c_r + \lambda_r z)|}= - \bm{n}^\phi(z)
		\end{equation*}
		and 
		\begin{equation*}
			z = \nabla \phi(\bm{n}^\phi(z)) = - \nabla \phi(\nabla \bm{\delta}^\phi_K(c_r + \lambda_r z)) = - \bm{\nu}^\phi_K(c_r + \lambda_r z)
		\end{equation*}
		for every  $ z \in \mathcal{W}^\phi $, whence we infer that
		\begin{equation*}
			c_r + (\lambda_r + r)\mathcal{W}^\phi = \{x - r\bm{\nu}^\phi_K(x): x \in S^\phi(K,r)\}\subseteq \partial K.
		\end{equation*}
		This implies that $c_r + (\lambda_r + r)\mathcal{W}^\phi = \partial K $ and the proof is complete.
	\end{proof}
	
	\begin{Remark}
		Similar inequalities have been proved for smooth domains in \cite{MR2514391} and for certain sets of finite perimeter in \cite{MR4160798}.
	\end{Remark}
	
	\begin{Theorem}\label{theo: Alexandrov}
		Suppose $ C \subseteq \mathbf{R}^{n+1} $ is a convex body, $ r= 1, \ldots , n $, $ \lambda > 0 $ and $ \mathcal{C}^\phi_{n-r}(C, \cdot) = \lambda \mathcal{C}^\phi_n(C, \cdot) $. Then $ \partial C = a + s \mathcal{W}^\phi $ for some $ a \in \mathbf{R}^{n+1} $ and $ s > 0 $.
	\end{Theorem}
	
	\begin{proof}
		The Newton-McLaurin inequality (see \cite[Theorem 1.1]{MR1786404}) together with Lemma \ref{lem: constant mean curvature properties} implies 
		\begin{equation}\label{theo: Alexandrov eq1}
			\frac{1}{n}\bm{H}^\phi_{C,1}(a,u)\geq \frac{\mathcal{C}^\phi_n(C)}{(n+1)\mathcal{L}^{n+1}(C)}\qquad \textrm{for $ \mathcal{H}^n $ a.e.\ $(a,u)\in \widetilde{N}_n^\phi(C) $.}
		\end{equation}

		For every $ \epsilon > 0 $ we set 
		\begin{equation*}
			Z_\epsilon = \Bigg\{ (a,u) \in  \widetilde{N}^\phi_n(C) :\frac{1}{n}\bm{H}^\phi_{C,1}(a,u)\geq (1+ \epsilon) \frac{\mathcal{C}^\phi_n(C)}{(n+1)\mathcal{L}^{n+1}(C)}   \Bigg\}.
		\end{equation*}
		We claim that $ \mathcal{H}^n(Z_\epsilon) =0 $ for every $ \epsilon > 0 $. If there existed $ \epsilon > 0 $ so that $ \mathcal{H}^n(Z_\epsilon) > 0 $, then we could employ \eqref{theo: Alexandrov eq1} to estimate
		\begin{flalign*}
			&	n	\int_{\widetilde{N}_n^\phi(C)} J^\phi_C(a,u)\frac{\phi(\bm{n}^\phi(u))}{\bm{H}^\phi_{C,1}(a,u)}\, d\mathcal{H}^n(a,u) \\
			& \qquad = 	n	\int_{\widetilde{N}_n^\phi(C) \setminus Z_\epsilon} J^\phi_C(a,u)\frac{\phi(\bm{n}^\phi(u))}{\bm{H}^\phi_{C,1}(a,u)}\, d\mathcal{H}^n(a,u) + 	n	\int_{Z_\epsilon} J^\phi_C(a,u)\frac{\phi(\bm{n}^\phi(u))}{\bm{H}^\phi_{C,1}(a,u)}\, d\mathcal{H}^n(a,u) \\
			& \qquad \leq \frac{(n+1)\mathcal{L}^{n+1}(C)}{\mathcal{C}^\phi_n(C)}\widetilde{\mathcal{C}}^\phi_n(C, N^\phi(C) \setminus Z_\epsilon) + \frac{(n+1)\mathcal{L}^{n+1}(C)}{\mathcal{C}^\phi_n(C)(1+ \epsilon)}\widetilde{\mathcal{C}}^\phi_n(C, Z_\epsilon)\\
			& \qquad  < (n+1)\mathcal{L}^{n+1}(C)
		\end{flalign*} 
		in contradiction with the inequality in Theorem \ref{theo: heintze karcher}. 
		
		Since $ \mathcal{H}^n(Z_\epsilon) =0 $ for every $ \epsilon > 0 $, we infer from \eqref{theo: Alexandrov eq1} that 
		\begin{equation*}
			\frac{1}{n}\bm{H}^\phi_{C,1}(a,u) = \frac{\mathcal{C}^\phi_n(C)}{(n+1)\mathcal{L}^{n+1}(C)}\qquad \textrm{for $ \mathcal{H}^n $ a.e.\ $(a,u)\in \widetilde{N}_n^\phi(C) $}
		\end{equation*}
		and we obtain the conclusion employing the second part of Theorem \ref{theo: heintze karcher}.
	\end{proof}


\begin{thebibliography}{HLMG09}
		
		\bibitem[Ale62]{MR143162}
		A.~D. Alexandrov.
		\newblock A characteristic property of spheres.
		\newblock {\em Ann. Mat. Pura Appl. (4)}, 58:303--315, 1962.
		
		\bibitem[AW21]{MR4214340}
		Ben Andrews and Yong Wei.
		\newblock Volume preserving flow by powers of the {$k$}-th mean curvature.
		\newblock {\em J. Differential Geom.}, 117(2):193--222, 2021.
		
		\bibitem[BS18]{MR3829571}
		Maria~Chiara Bertini and Carlo Sinestrari.
		\newblock Volume-preserving nonhomogeneous mean curvature flow of convex
		hypersurfaces.
		\newblock {\em Ann. Mat. Pura Appl. (4)}, 197(4):1295--1309, 2018.
		
		\bibitem[DRKS20]{MR4160798}
		Antonio De~Rosa, S\l~awomir Kolasi\'{n}ski, and Mario Santilli.
		\newblock Uniqueness of critical points of the anisotropic isoperimetric
		problem for finite perimeter sets.
		\newblock {\em Arch. Ration. Mech. Anal.}, 238(3):1157--1198, 2020.
		
		\bibitem[Fed59]{MR0110078}
		Herbert Federer.
		\newblock Curvature measures.
		\newblock {\em Trans. Amer. Math. Soc.}, 93:418--491, 1959.
		
		\bibitem[Fed69]{MR0257325}
		Herbert Federer.
		\newblock {\em Geometric measure theory}.
		\newblock Die Grundlehren der mathematischen Wissenschaften, Band 153.
		Springer-Verlag New York Inc., New York, 1969.
		
		\bibitem[HK78]{MR533065}
		Ernst Heintze and Hermann Karcher.
		\newblock A general comparison theorem with applications to volume estimates
		for submanifolds.
		\newblock {\em Ann. Sci. \'{E}cole Norm. Sup. (4)}, 11(4):451--470, 1978.
		
		\bibitem[HL00]{MR1782274}
		Daniel Hug and G\"{u}nter Last.
		\newblock On support measures in {M}inkowski spaces and contact distributions
		in stochastic geometry.
		\newblock {\em Ann. Probab.}, 28(2):796--850, 2000.
		
		\bibitem[HLMG09]{MR2514391}
		Yijun He, Haizhong Li, Hui Ma, and Jianquan Ge.
		\newblock Compact embedded hypersurfaces with constant higher order anisotropic
		mean curvatures.
		\newblock {\em Indiana Univ. Math. J.}, 58(2):853--868, 2009.
		
		\bibitem[HLW04]{MR2031455}
		Daniel Hug, G{\"u}nter Last, and Wolfgang Weil.
		\newblock A local {S}teiner-type formula for general closed sets and
		applications.
		\newblock {\em Math. Z.}, 246(1-2):237--272, 2004.
		
		\bibitem[HM87]{MR884486}
		Gary Harris and Clyde Martin.
		\newblock The roots of a polynomial vary continuously as a function of the
		coefficients.
		\newblock {\em Proc. Amer. Math. Soc.}, 100(2):390--392, 1987.
		
		\bibitem[Hsi54]{MR68236}
		Chuan-Chih Hsiung.
		\newblock Some integral formulas for closed hypersurfaces.
		\newblock {\em Math. Scand.}, 2:286--294, 1954.
		
		\bibitem[Koh95]{MR1310959}
		Peter Kohlmann.
		\newblock Compact convex bodies with one curvature measure near the surface
		measure.
		\newblock {\em J. Reine Angew. Math.}, 458:201--217, 1995.
		
		\bibitem[Koh98]{MR1604003}
		Peter Kohlmann.
		\newblock Characterizations via linear combinations of curvature measures.
		\newblock {\em Arch. Math. (Basel)}, 70(3):250--256, 1998.
		
		\bibitem[KS21]{kolasinski2021regularity}
		Sławomir Kolasiński and Mario Santilli.
		\newblock Regularity of distance functions from arbitrary closed sets, 2021.
		
		\bibitem[MR91]{MR1173047}
		Sebasti\'{a}n Montiel and Antonio Ros.
		\newblock Compact hypersurfaces: the {A}lexandrov theorem for higher order mean
		curvatures.
		\newblock In {\em Differential geometry}, volume~52 of {\em Pitman Monogr.
			Surveys Pure Appl. Math.}, pages 279--296. Longman Sci. Tech., Harlow, 1991.
		
		\bibitem[Nic00]{MR1786404}
		Constantin~P. Niculescu.
		\newblock A new look at {N}ewton's inequalities.
		\newblock {\em JIPAM. J. Inequal. Pure Appl. Math.}, 1(2):Article 17, 14, 2000.
		
		\bibitem[Ros87]{MR996826}
		Antonio Ros.
		\newblock Compact hypersurfaces with constant higher order mean curvatures.
		\newblock {\em Rev. Mat. Iberoamericana}, 3(3-4):447--453, 1987.
		
		\bibitem[Ros88]{MR925120}
		Antonio Ros.
		\newblock Compact hypersurfaces with constant scalar curvature and a congruence
		theorem.
		\newblock {\em J. Differential Geom.}, 27(2):215--223, 1988.
		\newblock With an appendix by Nicholas J. Korevaar.
		
		\bibitem[San20a]{MR4117503}
		Mario Santilli.
		\newblock Fine properties of the curvature of arbitrary closed sets.
		\newblock {\em Ann. Mat. Pura Appl. (4)}, 199(4):1431--1456, 2020.
		
		\bibitem[San20b]{santilli2020uniqueness}
		Mario Santilli.
		\newblock Uniqueness of singular convex hypersurfaces with lower bounded k-th
		mean curvature, 2020.
		
		\bibitem[San21]{MR4279967}
		Mario Santilli.
		\newblock Distance functions with dense singular sets.
		\newblock {\em Comm. Partial Differential Equations}, 46(7):1319--1325, 2021.
		
		\bibitem[Sch79]{MR522031}
		Rolf Schneider.
		\newblock Bestimmung konvexer {K}\"{o}rper durch {K}r\"{u}mmungsmasse.
		\newblock {\em Comment. Math. Helv.}, 54(1):42--60, 1979.
		
		\bibitem[Sch14]{MR3155183}
		Rolf Schneider.
		\newblock {\em Convex bodies: the {B}runn-{M}inkowski theory}, volume 151 of
		{\em Encyclopedia of Mathematics and its Applications}.
		\newblock Cambridge University Press, Cambridge, expanded edition, 2014.
		
		\bibitem[Sin15]{MR3396440}
		Carlo Sinestrari.
		\newblock Convex hypersurfaces evolving by volume preserving curvature flows.
		\newblock {\em Calc. Var. Partial Differential Equations}, 54(2):1985--1993,
		2015.
		
		\bibitem[Z{\"a}h86]{MR849863}
		M.~Z{\"a}hle.
		\newblock Integral and current representation of {F}ederer's curvature
		measures.
		\newblock {\em Arch. Math. (Basel)}, 46(6):557--567, 1986.
		
	\end{thebibliography}
	
	\medskip 
	
	\noindent Institut f\"ur Mathematik, Universit\"at Augsburg, \newline Universit\"atsstr.\ 14, 86159, Augsburg, Germany,
	\newline mario.santilli@math.uni-augsburg.de
	
	\noindent \textbf{New affiliation} 
	
	Mario Santilli,  Department of Information Engineering, Computer Science and Mathematics, Università degli Studi dell'Aquila,
	67100 L’Aquila, Italy, mario.santilli@univaq.it
\end{document}